\documentclass[12pt]{amsart}
\usepackage[utf8]{inputenc}
\usepackage{amsmath, amsthm, amsfonts, amssymb, graphicx, xcolor, fullpage}
\usepackage{comment, pst-all, pstricks}
\usepackage{float, subfigure, listings, csquotes, physics, tikz, mathdots, yhmath, cancel, siunitx, array, multirow, gensymb, tabularx, extarrows, booktabs, setspace}
\usepackage[colorlinks=true, allcolors=blue]{hyperref}
\usepackage{verbatim}

\usetikzlibrary{fadings, patterns, shadows.blur, shapes}

\newtheorem{teo}{Theorem}[section]

\newtheorem{propo}[teo]{Proposition}

\newtheorem{defi}[teo]{Definition}

\begin{document}

\title{Bridging Graph-Theoretical and Topological Approaches: Connectivity and Jordan Curves in the Digital Plane}

\author{Yazmin Cote}
\address{Escuela de Matem\'aticas, Universidad Industrial de Santander,  Bucaramanga - Colombia.}
\email{yazmin2228080@correo.uis.edu.co}

\author{Carlos Uzc\'ategui-Aylwin}
\address{Escuela de Matem\'aticas, Universidad Industrial de Santander, Bucaramanga - Colombia.}
\email{cuzcatea@saber.uis.edu.co}

\date{}

\begin{abstract}
This article explores the connections between graph-theoretical and topological approaches in the study of the Jordan curve theorem for grids. Building on the foundational work of Rosenfeld, who developed adjacency-based concepts on $\mathbb{Z}^2$, and the subsequent introduction of the topological digital plane $\mathbb{K}^2$ with  the Khalimsky topology by Khalimsky, Kopperman, and Meyer, we investigate the interplay between these perspectives. Inspired  by the work of  Khalimsky, Kopperman, and Meyer, we define an operator 
$\Gamma^*$  transforming subsets of $\mathbb{Z}^2$ into subsets of $\mathbb{K}^2$. This operator is essential  for demonstrating how 8-paths, 4-connectivity, and other discrete structures in $\mathbb{Z}^2$ correspond to topological properties in 
$\mathbb{K}^2$. Moreover, we address whether the topological Jordan curve theorem for $\mathbb{K}^2$
can be derived from the graph-theoretical version on $\mathbb{Z}^2$. Our results illustrate the deep and intricate relationship between these two methodologies, shedding light on their complementary roles in digital topology.
\end{abstract} 



\maketitle

\section{ Introduction}

The study of the topological properties of sets of grid points in $\mathbb{Z}^2$ (or in a higher-dimensional grid) originated with the work of Rosenfeld \cite{R1,R2}, who introduced graph-theoretical notions based on adjacency relations among points in $\mathbb{Z}^2$. His primary result was a Jordan curve theorem adapted to this setting. There is a vast literature about this topic, see for instance \cite{fajardoJonard2021,davis,kong,slapal2004,slapal2006} and reference there in. Rosenfeld's approach did not rely on an underlying topology on $\mathbb{Z}^2$; this was later achieved by Khalimsky, Kopperman, and Meyer \cite{Khalimsky, Khalimsky2}, who introduced a topology on $\mathbb{Z}$ (known as the Khalimsky topology) and proved a topological Jordan curve theorem for $\mathbb{Z}^2$ (with the product topology). Their proof of the topological Jordan curve theorem does not depend on Rosenfeld's theorem. However, as shown by Khalimsky, Kopperman, and Meyer \cite{Khalimsky2}, Rosenfeld's result can, in fact, be derived using the topological version of the Jordan curve theorem, revealing an interesting connection between the graph-theoretical and topological approaches. The main purpose of this article is to further explore the similarities between these two approaches.

The \textbf{Rosenfeld plane} is the grid $\mathbb{Z}^{2}$ equipped with two adjacency relations. The \textbf{8-adjacent} points to $p = (x, y) \in \mathbb{Z}^2$ are the points of the form $(x \pm 1, y \pm 1)$ together with the \textbf{4-adjacent} points, which are those of the form $(x \pm 1, y)$ and $(x, y \pm 1)$ (see Figure \ref{adyacencia}).

\begin{figure}[H]
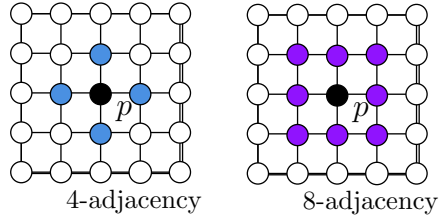

\centering

\tikzset{every picture/.style={line width=0.5pt}} 


    \caption{$N_4(p)$ and  $N_8(p)$.}
    \label{adyacencia}
\end{figure}

These notions allow for the natural introduction of the concept of a $k$-path and the corresponding notion of $k$-connectivity for $k=4,8$ (see section \ref{SeccionPlanoR} for the exact definitions). Rosenfeld (\cite{R1, R2}) proved the following result.

\begin{teo} 
\label{Jordank-intro}
Let $\{k, k^{\prime}\}=\{4,8\}$ and let $J$ be a closed $k$-curve. Then $\mathbb{Z}^2 \setminus J$ has two $k^{\prime}$-connected components. 
\end{teo}

The use of both adjacency relations may seem problematic, but it is necessary due to counterintuitive examples where the Jordan curve theorem fails if only one of them is applied (see \cite{KongRosenfeld1989}). This issue, however, does not arise in the topological approach.

The  {\em digital topology} (or {\em  Khalimsky topology}) on  $\mathbb{Z}$ is generated by the following sets: 
$$
N( n)=
\begin{cases}\{n\} & \text { If } n \text { is odd. } \\ \{n-1,n,n+1 \} & \text {If } n \text { is even.}
\end{cases}
$$
We denote by $\mathbb{K}$ the digital line, i.e., $\mathbb{Z}$ with the digital topology. It is an connected Alexandroff space where $N(n)$ is the minimal open set containing $n$ for each $n \in \mathbb{Z}$. The finite intervals of $\mathbb{K}$ are analogous to the compact intervals of $\mathbb{R}$ and serve to define curves on the digital plane $\mathbb{K}^2$ with its product topology. Specifically, $C \subseteq \mathbb{K}^2$ is an arc if, with the subspace topology, it is homeomorphic to a finite interval of $\mathbb{K}$, and $J \subseteq \mathbb{K}^2$ is a Jordan curve if it is connected and $J \setminus \{a\}$ is an arc for every $a \in J$ (more information is given in Section \ref{SeccionTopK}). Khalimsky, Kopperman, and Mayer \cite{Khalimsky} proved the following:

\begin{teo} 
\label{JordanTOP-intro} 
If $J$ is a Jordan curve on $\mathbb{K}^2$, then $\mathbb{K}^2 \setminus J$ has two connected components. 
\end{teo}

As mentioned previously, Khalimsky et al. \cite{Khalimsky2} derived Theorem \ref{Jordank-intro} from Theorem \ref{JordanTOP-intro}. We present a sketch of their argument as motivation for our work. The key idea was to introduce the map $\Gamma: \mathbb{Z}^2 \rightarrow \mathbb{K}^2$ defined by $\Gamma(x, y) = (x + y, y - x)$. While it is not true that $\Gamma$ transforms a closed $k$-curve on $\mathbb{Z}^2$ into a Jordan curve on $\mathbb{K}^2$, given a closed $k$-curve $J$ on $\mathbb{Z}^2$, they found a way to expand $\Gamma(J)$ to a larger set, $\Gamma(J)^*$, such that $\Gamma(J)^*$ is a Jordan curve, allowing Theorem \ref{JordanTOP-intro} to be applied to obtain the two components, $A$ and $B$, of $\mathbb{K}^{2} \setminus \Gamma(J)^*$. Finally, they showed that $\Gamma^{-1}(A)$ and $\Gamma^{-1}(B)$ are the two $k'$-components of $\mathbb{Z}^2 \setminus J$.

Inspired by their argument, we define an operator $\Gamma^*$ that  transforms subsets of $\mathbb{Z}^2$ into subsets of $\mathbb{K}^2$ (see \eqref{defoperador*} for the definition). This operator  will play a crucial role in our work. 
The following are examples of the type of results we are presenting (see sections \ref{sec-preser-curvas} and \ref{sec-prese-conex}):

\begin{itemize}
\item[-]
$C$ is an 8-path in $\mathbb{Z}^2$ if and only if $\Gamma^{*}(C)$ is an arc in $\mathbb{K}^2$.

\item[-] 
For an arc  $D$ in $\mathbb{K}^2$,  $\Gamma^{-1}(D)$ is an 8-path in $\mathbb{Z}^2$ if and only if $\Gamma^{*}(\Gamma^{-1}(D)) = D$.

\item[-] 
$A \subseteq \mathbb{Z}^{2}$ is 4-connected if and only if $\Gamma(A)$ is connected in $\mathbb{K}^{2}$.

\item[-] 
$A \subseteq \mathbb{Z}^{2}$ is 8-connected if and only if $\Gamma^{*}(A)$ is connected in $\mathbb{K}^{2}$. 
\end{itemize}

\noindent The simplicity of the statements conceals the intricate and extensive case-by-case arguments used in its proofs.

A natural question is whether Theorem \ref{JordanTOP-intro} can be proved using Theorem \ref{Jordank-intro}. In Section \ref{sec-curvaJordan}, we explore this issue, which served as our primary motivation. Using Theorem \ref{Jordank-intro}, we show that if $J$ is a Jordan curve in $ \mathbb{K}^2$ with more than five points and either $J \subseteq \Gamma(\mathbb{Z}^2)$ or $\Gamma^{*}(\Gamma^{-1}(J)) = J$, then $\mathbb{K}^2 \setminus J$ has two components.

Our results reveal a strong connection between both approaches: the graph-theoretical on $\mathbb{Z}^2$ and the topological  on $\mathbb{K}^2$,  shedding light on their complementary roles in digital topology.

\section{ Rosenfeld plane}
\label{SeccionPlanoR}

The \textbf{Rosenfeld plane} is $\mathbb{Z}^{2}$, equipped with specific adjacency relations. The points that are {\bf $8$-adjacent} to a point $p = (x, y) \in \mathbb{Z}^2$ are those of the form $(x \pm 1, y \pm 1)$ together with the  {\bf $4$-adjacent} points that are those of the form $(x \pm 1, y)$ and $(x, y \pm 1)$. The set of points that are $8$-adjacent to $p$ (respectively, $4$-adjacent) is denoted by $N_8(p)$ (respectively, $N_4(p)$). These sets are illustrated in Figure \ref{adyacencia}.

Adjacencies allow for the introduction of a notion of connectivity for subsets of $\mathbb{Z}^2$ that does not rely on topological considerations. We say that two subsets $A$ and $B$ of $\mathbb{Z}^2$ are \textbf{4-adjacent} if there exist $x \in A$ and $z \in B$ such that $x$ and $z$ are 4-adjacent.  

A subset $S$ of $\mathbb{Z}^2$ is \textbf{4-disconnected} if it can be divided into two nonempty subsets that are not 4-adjacent to each other. Otherwise, $S$ is said to be \textbf{4-connected}. A \textbf{4-component} of a subset $S$ of $\mathbb{Z}^2$ is a maximal 4-connected subset of $S$. Analogously, we define the notions of an \textbf{8-connected set} and an \textbf{8-component}.  

Adjacencies also allow for defining the concept of a path. Let $C$ be a subset of $\mathbb{Z}^2$. A point in $C$ is said to be \textbf{8-endpoint} if it has exactly one point in $C$ that is 8-adjacent to it. Analogously, we define when a point is \textbf{4-endpoint}.  

\begin{defi}  
A finite subset $C \subseteq \mathbb{Z}^2$ is an \textbf{8-path} if it has exactly two 8-endpoints, and all other points in $C$ have exactly two points in $C$ that are 8-adjacent. Analogously, we define a \textbf{4-path}.  
\end{defi}  

In other words, an $8$-path is a set $\{x_0, x_1, x_2, \ldots, x_n\}$ of points in $\mathbb{Z}^2$ such that $x_i$ and $x_j$, with $i < j \leq n$, are 8-adjacent if and only if $j = i+1$.  The following proposition is useful for determining  when a set is $k$-connected. 

\begin{propo}  
\label{curvasimple}  
A subset $A \subseteq \mathbb{Z}^2$ is $8$-connected (respectively, $4$-connected) if and only if, for all $x, y \in A$, there exists an $8$-path (respectively, $4$-path) with $x$ and $y$ as the 8-endpoint points (respectively, 4-endpoint).  
\end{propo}  

\begin{defi}  
A finite subset $C \subseteq \mathbb{Z}^2$ is a \textbf{closed $8$-curve} if every point in $C$ has exactly two points in $C$ that are 8-adjacent to it. Analogously, we define a \textbf{closed $4$-curve}.  
\end{defi}  

The following result is the Jordan curve theorem proven by Rosenfeld (\cite{R1,R2}). The statement we use below is as it appears in \cite{kong}.

\begin{teo} (Rosenfeld's Jordan curve Theorem) \label{Jordank} Let $\{k, k^{\prime}\} = \{4, 8\}$. If $J$ is a  closed $k$-curve with more than 4 points, then: \begin{enumerate} \item $\mathbb{Z}^2 \setminus J$ has two $k^{\prime}$-components.

\item Each point of $J$ has a $k'$-neighbor in each of the two $k'$-components of $\mathbb{Z}^2 \setminus J$. \end{enumerate}
\end{teo}

In the following result, we will use the enumeration of the 8-neighborhood of a point $p \in \mathbb{Z}^2$ as shown in Figure \ref{ejemplosimple}.

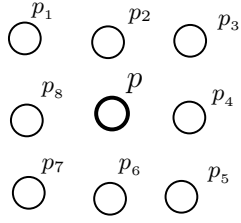
\begin{figure}[H]
  \centering
\tikzset{every picture/.style={line width=0.75pt}} 

\begin{tikzpicture}[x=0.75pt,y=0.75pt,yscale=-1,xscale=1]

\draw  [fill={rgb, 255:red, 255; green, 255; blue, 255 }  ,fill opacity=1 ] (164,91.79) .. controls (164,87.37) and (167.58,83.78) .. (172.01,83.78) .. controls (176.43,83.78) and (180.01,87.37) .. (180.01,91.79) .. controls (180.01,96.21) and (176.43,99.79) .. (172.01,99.79) .. controls (167.58,99.79) and (164,96.21) .. (164,91.79) -- cycle ;
\draw  [color={rgb, 255:red, 0; green, 0; blue, 0 }  ,draw opacity=1 ][fill={rgb, 255:red, 255; green, 255; blue, 255 }  ,fill opacity=1 ][line width=0.75]  (81.01,51.87) .. controls (81.01,47.44) and (84.59,43.86) .. (89.01,43.86) .. controls (93.43,43.86) and (97.02,47.44) .. (97.02,51.87) .. controls (97.02,56.29) and (93.43,59.87) .. (89.01,59.87) .. controls (84.59,59.87) and (81.01,56.29) .. (81.01,51.87) -- cycle ;
\draw  [color={rgb, 255:red, 0; green, 0; blue, 0 }  ,draw opacity=1 ][fill={rgb, 255:red, 255; green, 255; blue, 255 }  ,fill opacity=1 ][line width=0.75]  (164.41,53.06) .. controls (164.41,48.64) and (168,45.05) .. (172.42,45.05) .. controls (176.84,45.05) and (180.42,48.64) .. (180.42,53.06) .. controls (180.42,57.48) and (176.84,61.07) .. (172.42,61.07) .. controls (168,61.07) and (164.41,57.48) .. (164.41,53.06) -- cycle ;
\draw  [color={rgb, 255:red, 0; green, 0; blue, 0 }  ,draw opacity=1 ][fill={rgb, 255:red, 255; green, 255; blue, 255 }  ,fill opacity=1 ][line width=0.75]  (83,129.79) .. controls (83,125.37) and (86.58,121.78) .. (91.01,121.78) .. controls (95.43,121.78) and (99.01,125.37) .. (99.01,129.79) .. controls (99.01,134.21) and (95.43,137.79) .. (91.01,137.79) .. controls (86.58,137.79) and (83,134.21) .. (83,129.79) -- cycle ;
\draw  [color={rgb, 255:red, 0; green, 0; blue, 0 }  ,draw opacity=1 ][fill={rgb, 255:red, 255; green, 255; blue, 255 }  ,fill opacity=1 ][line width=1.5]  (125,89.79) .. controls (125,85.37) and (128.58,81.78) .. (133.01,81.78) .. controls (137.43,81.78) and (141.01,85.37) .. (141.01,89.79) .. controls (141.01,94.21) and (137.43,97.79) .. (133.01,97.79) .. controls (128.58,97.79) and (125,94.21) .. (125,89.79) -- cycle ;
\draw  [color={rgb, 255:red, 0; green, 0; blue, 0 }  ,draw opacity=1 ][fill={rgb, 255:red, 255; green, 255; blue, 255 }  ,fill opacity=1 ] (160,131.79) .. controls (160,127.37) and (163.58,123.78) .. (168.01,123.78) .. controls (172.43,123.78) and (176.01,127.37) .. (176.01,131.79) .. controls (176.01,136.21) and (172.43,139.79) .. (168.01,139.79) .. controls (163.58,139.79) and (160,136.21) .. (160,131.79) -- cycle ;
\draw  [color={rgb, 255:red, 0; green, 0; blue, 0 }  ,draw opacity=1 ][fill={rgb, 255:red, 255; green, 255; blue, 255 }  ,fill opacity=1 ][line width=0.75]  (123,53.79) .. controls (123,49.37) and (126.58,45.78) .. (131.01,45.78) .. controls (135.43,45.78) and (139.01,49.37) .. (139.01,53.79) .. controls (139.01,58.21) and (135.43,61.79) .. (131.01,61.79) .. controls (126.58,61.79) and (123,58.21) .. (123,53.79) -- cycle ;
\draw  [color={rgb, 255:red, 0; green, 0; blue, 0 }  ,draw opacity=1 ][fill={rgb, 255:red, 255; green, 255; blue, 255 }  ,fill opacity=1 ][line width=0.75]  (82.02,93.25) .. controls (82.02,88.83) and (85.61,85.25) .. (90.03,85.25) .. controls (94.45,85.25) and (98.03,88.83) .. (98.03,93.25) .. controls (98.03,97.67) and (94.45,101.26) .. (90.03,101.26) .. controls (85.61,101.26) and (82.02,97.67) .. (82.02,93.25) -- cycle ;
\draw  [color={rgb, 255:red, 0; green, 0; blue, 0 }  ,draw opacity=1 ][fill={rgb, 255:red, 255; green, 255; blue, 255 }  ,fill opacity=1 ] (124,132.79) .. controls (124,128.37) and (127.58,124.78) .. (132.01,124.78) .. controls (136.43,124.78) and (140.01,128.37) .. (140.01,132.79) .. controls (140.01,137.21) and (136.43,140.79) .. (132.01,140.79) .. controls (127.58,140.79) and (124,137.21) .. (124,132.79) -- cycle ;

\draw (139.01,68.19) node [anchor=north west][inner sep=0.75pt]    {$p$};
\draw (91.02,32.27) node [anchor=north west][inner sep=0.75pt]  [font=\scriptsize]  {$p_{_{1}}$};
\draw (140.01,35.19) node [anchor=north west][inner sep=0.75pt]  [font=\scriptsize]  {$p_{2}$};
\draw (185.02,38.27) node [anchor=north west][inner sep=0.75pt]  [font=\scriptsize]  {$p_{_{3}}$};
\draw (182.01,76.19) node [anchor=north west][inner sep=0.75pt]  [font=\scriptsize]  {$p_{_{4}}$};
\draw (180.01,115.19) node [anchor=north west][inner sep=0.75pt]  [font=\scriptsize]  {$p_{_{5}}$};
\draw (135.01,111.19) node [anchor=north west][inner sep=0.75pt]  [font=\scriptsize]  {$p_{_{6}}$};
\draw (96.01,111.19) node [anchor=north west][inner sep=0.75pt]  [font=\scriptsize]  {$p_{7}$};
\draw (96.02,72.27) node [anchor=north west][inner sep=0.75pt]  [font=\scriptsize]  {$p_{_{8}}$};

\end{tikzpicture}
    \caption{Labeling of the 8-neighborhood points of  $p$.}
    \label{ejemplosimple}
\end{figure}

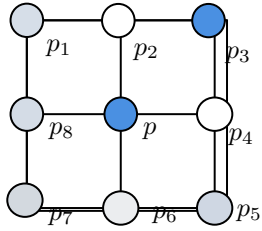
\begin{figure}[H]
    \centering

\tikzset{every picture/.style={line width=0.75pt}} 

\begin{tikzpicture}[x=0.60pt,y=0.60pt,yscale=-1,xscale=1]

\draw  [draw opacity=0][fill={rgb, 255:red, 255; green, 255; blue, 255 }  ,fill opacity=1 ] (53.06,20.88) -- (179.26,20.88) -- (179.26,141.11) -- (53.06,141.11) -- cycle ; \draw   (53.06,20.88) -- (53.06,141.11)(112.27,20.88) -- (112.27,141.11)(171.48,20.88) -- (171.48,141.11) ; \draw   (53.06,20.88) -- (179.26,20.88)(53.06,80.09) -- (179.26,80.09)(53.06,139.3) -- (179.26,139.3) ; \draw    ;
\draw  [color={rgb, 255:red, 0; green, 0; blue, 0 }  ,draw opacity=1 ][fill={rgb, 255:red, 207; green, 215; blue, 227 }  ,fill opacity=1 ][line width=0.75]  (160.35,139.3) .. controls (160.35,133.49) and (165.34,128.78) .. (171.48,128.78) .. controls (177.63,128.78) and (182.61,133.49) .. (182.61,139.3) .. controls (182.61,145.12) and (177.63,149.83) .. (171.48,149.83) .. controls (165.34,149.83) and (160.35,145.12) .. (160.35,139.3) -- cycle ;
\draw  [color={rgb, 255:red, 0; green, 0; blue, 0 }  ,draw opacity=1 ][fill={rgb, 255:red, 255; green, 255; blue, 255 }  ,fill opacity=1 ] (160.35,80.09) .. controls (160.35,74.28) and (165.34,69.57) .. (171.48,69.57) .. controls (177.63,69.57) and (182.61,74.28) .. (182.61,80.09) .. controls (182.61,85.9) and (177.63,90.62) .. (171.48,90.62) .. controls (165.34,90.62) and (160.35,85.9) .. (160.35,80.09) -- cycle ;
\draw  [fill={rgb, 255:red, 214; green, 222; blue, 232 }  ,fill opacity=1 ] (42.79,20.88) .. controls (42.79,15.07) and (47.39,10.36) .. (53.06,10.36) .. controls (58.73,10.36) and (63.33,15.07) .. (63.33,20.88) .. controls (63.33,26.69) and (58.73,31.4) .. (53.06,31.4) .. controls (47.39,31.4) and (42.79,26.69) .. (42.79,20.88) -- cycle ;
\draw  [fill={rgb, 255:red, 212; green, 220; blue, 230 }  ,fill opacity=1 ] (42.79,80.09) .. controls (42.79,74.28) and (47.39,69.57) .. (53.06,69.57) .. controls (58.73,69.57) and (63.33,74.28) .. (63.33,80.09) .. controls (63.33,85.9) and (58.73,90.61) .. (53.06,90.61) .. controls (47.39,90.61) and (42.79,85.9) .. (42.79,80.09) -- cycle ;
\draw  [fill={rgb, 255:red, 74; green, 144; blue, 226 }  ,fill opacity=1 ] (102,80.09) .. controls (102,74.28) and (106.6,69.57) .. (112.27,69.57) .. controls (117.94,69.57) and (122.54,74.28) .. (122.54,80.09) .. controls (122.54,85.9) and (117.94,90.61) .. (112.27,90.61) .. controls (106.6,90.61) and (102,85.9) .. (102,80.09) -- cycle ;
\draw  [fill={rgb, 255:red, 74; green, 144; blue, 226 }  ,fill opacity=1 ] (157.33,21.05) .. controls (157.33,15.24) and (161.93,10.53) .. (167.6,10.53) .. controls (173.27,10.53) and (177.87,15.24) .. (177.87,21.05) .. controls (177.87,26.86) and (173.27,31.57) .. (167.6,31.57) .. controls (161.93,31.57) and (157.33,26.86) .. (157.33,21.05) -- cycle ;
\draw  [fill={rgb, 255:red, 255; green, 255; blue, 255 }  ,fill opacity=1 ] (100.47,21.05) .. controls (100.47,15.24) and (105.07,10.53) .. (110.74,10.53) .. controls (116.41,10.53) and (121.01,15.24) .. (121.01,21.05) .. controls (121.01,26.86) and (116.41,31.57) .. (110.74,31.57) .. controls (105.07,31.57) and (100.47,26.86) .. (100.47,21.05) -- cycle ;
\draw  [color={rgb, 255:red, 0; green, 0; blue, 0 }  ,draw opacity=1 ][fill={rgb, 255:red, 235; green, 237; blue, 239 }  ,fill opacity=1 ] (101.14,139.3) .. controls (101.14,133.49) and (106.12,128.78) .. (112.27,128.78) .. controls (118.42,128.78) and (123.4,133.49) .. (123.4,139.3) .. controls (123.4,145.12) and (118.42,149.83) .. (112.27,149.83) .. controls (106.12,149.83) and (101.14,145.12) .. (101.14,139.3) -- cycle ;
\draw  [color={rgb, 255:red, 0; green, 0; blue, 0 }  ,draw opacity=1 ][fill={rgb, 255:red, 211; green, 217; blue, 223 }  ,fill opacity=1 ] (40.74,134.36) .. controls (40.74,128.55) and (45.73,123.84) .. (51.87,123.84) .. controls (58.02,123.84) and (63,128.55) .. (63,134.36) .. controls (63,140.18) and (58.02,144.89) .. (51.87,144.89) .. controls (45.73,144.89) and (40.74,140.18) .. (40.74,134.36) -- cycle ;

\draw (124.54,83.49) node [anchor=north west][inner sep=0.75pt]  [font=\footnotesize]  {$p$};
\draw (63.54,31.49) node [anchor=north west][inner sep=0.75pt]  [font=\footnotesize]  {$p_{1}$};
\draw (118.54,33.49) node [anchor=north west][inner sep=0.75pt]  [font=\footnotesize]  {$p_{2}$};
\draw (176.54,35.49) node [anchor=north west][inner sep=0.75pt]  [font=\footnotesize]  {$p_{3}$};
\draw (178.54,87.49) node [anchor=north west][inner sep=0.75pt]  [font=\footnotesize]  {$p_{4}$};
\draw (183.54,135.49) node [anchor=north west][inner sep=0.75pt]  [font=\footnotesize]  {$p_{5}$};
\draw (130.54,136.49) node [anchor=north west][inner sep=0.75pt]  [font=\footnotesize]  {$p_{6}$};
\draw (65,137.76) node [anchor=north west][inner sep=0.75pt]  [font=\footnotesize]  {$p_{7}$};
\draw (65.33,83.49) node [anchor=north west][inner sep=0.75pt]  [font=\footnotesize]  {$p_{8}$};

\end{tikzpicture}
    \caption{$p_2$ and $p_4$ are in different 4-components.}
    \label{pordondepasa}
\end{figure}
\begin{propo} 
\label{diferentescomponentes} 
Let $C$ be a closed 8-curve, $p \in C$ and $p_i \in C$ be a point 8-adjacent to $p$ that is not 4-adjacent to $p$ (i.e., $i \in \{1, 3, 5, 7\}$). Then $p_{i-1}$ and $p_{i+1}$ are in different 4-components of $\mathbb{Z}^2 \setminus C$ (where $p_{i-1}$ is $p_8$ if $i = 1$).
\end{propo}

\begin{proof}
We will assume, without loss of generality, that $p_i= p_3$. We will show that $p_2$ and $p_4$ are in different 4-components of $\mathbb{Z}^2 \setminus C$. By Theorem \ref{Jordank}, we know that $\mathbb{Z}^2 \setminus C$ has two 4-components, which we will denote as $A$ and $B$. Observe that $p_2$ and $p_4$ are not in $C$. We will assume that $p_4 \in A$ and show that $p_2 \in B$. Since $C$ is an closed 8-curve, it must pass through one of the following points: $p_1, p_8, p_7, p_6, p_5$ (see Figure \ref{pordondepasa}). We consider the following cases:

\medskip 

\noindent\textbf{Case 1:} $C$ passes through $p_1$ or $p_5$. Without loss of generality, suppose that $p_5 \in C$, as the reasoning is analogous for $p_1$. Since $p_4 \in A$, by part 2 of Theorem \ref{Jordank} (applied at $p$), we have that $p_j \in B$ for some $j \in \{2, 6, 8\}$. Let $D = \{p_1, p_2, p_6, p_7, p_8\}$. Note that $D \cap C = \emptyset$ (since $C$ is an closed 8-curve). Furthermore, $D = \{p_1, p_2, p_6, p_7, p_8\}$ is 4-connected, and since $D \cap B \neq \emptyset$, it follows that $D \subseteq B$. In particular, $p_2 \in B$.

\medskip

\noindent\textbf{Case 2:} Suppose that $C$ passes through $p_6$. Let $D = \{p_1, p_2, p_8\}$. Reasoning as in Case 1, we conclude that $D \subseteq B$.

\medskip

\noindent\textbf{Case 3:} Suppose that $C$ passes through $p_8$ or $p_7$. Using similar reasoning to the previous cases, we conclude that $p_4 \in A$ and $p_2 \in B$. \end{proof}

\section{ The topological digital plane}
\label{SeccionTopK}

A topological space $X$ is called {\bf Alexandroff} if the intersection of any collection of open sets is open. Clearly, every finite topological space is Alexandroff. In Alexandroff spaces, every point has a minimal open neighborhood, denoted by $N(x)$. The closure of the set $\{x\}$ is denoted by $cl(x)$ and we have that $y \in N(x)$ if and only if $x \in cl(y)$. A general reference on Alexandroff spaces is \cite[chapter 8]{Richmond}.

As mentioned in the introduction, the {\bf digital topology} (or \textbf{Khalimsky} topology) on $\mathbb{Z}$ is generated by the following sets:

$$
N( n)=
\begin{cases}\{n\} & \text { If } n \text { is odd. } \\ \{n-1,n,n+1 \} & \text { If } n \text { is even.}
\end{cases}
$$
We denote this topological space by the symbol $\mathbb{K}$. It is immediate that $\mathbb{K}$ is an Alexandroff space, since $N(n)$ is the minimal open neighborhood of $n$. Furthermore, $\mathbb{K}$ is a connected space. The pioneering article \cite{Khalimsky} serves as a general reference for the basic properties of the space $\mathbb{K}$.

The {\bf (topological)  digital plane} is $\mathbb{K}^2$ with the product topology. The points of $\mathbb{K}^2$ are classified into two types: a point $(x, y) \in \mathbb{K}^2$ is \textbf{pure} if both $x$ and $y$ are either closed points (i.e., even integers) or open points (i.e., odd integers) of $\mathbb{K}$; otherwise, we say that $(x, y)$ is \textbf{mixed}.

Connectivity in Alexandroff spaces is based on the concept of adjacent points, which we define as follows.

\begin{defi}
\label{defi_adyacencia}
Let $X$ be an Alexandroff space. The adjacency of a point $x \in X$ is the set
$$
A(x)=\{y\in X:y\neq x, \{x,y\} \textrm{ is connected} \}.
$$
\end{defi}

The following result characterizes the intervals in the digital line (see \cite[Theorem 3.2]{Khalimsky}).

\begin{teo} 
\label{arcosycaminos}
\begin{enumerate} 
\item If $X$ is a finite interval of the digital line, then there exist distinct points $x$ and $y$ such that $|A(y)| = |A(x)| = 1$ and $|A(w)| = 2$ for $w \in X \setminus \{x, y\}$.

\item If $X$ is a finite, connected topological space that contains distinct points $x$ and $y$ such that $|A(y)| = |A(x)| = 1$ and $|A(w)| = 2$ for $w \in X \setminus\{x, y\}$, then $X$ is homeomorphic to an interval of the digital line. \end{enumerate} \end{teo}

The following characterization of adjacency in an Alexandroff space is easy to verify (see \cite[Theorem 3.2]{Khalimsky}).

\begin{propo} \label{arcosycaminos1} Let $X$ be an Alexandroff space. For every $x \in X$, the following holds:
$$
A(x)\cup\{x\}=cl(x)\cup N(x).
$$
\end{propo}

In the particular case of the digital plane $\mathbb{K}^2$, the adjacency of points is as follows (see \cite[Lemma 4.2]{Khalimsky}). If $(x,y)$ is a pure point, then
\[
A(x,y)=\{x-1,x, x+1\}\times \{y-1,y, y+1\}\setminus\{(x,y)\}.
\]
If $(x,y)$ is a mixed point, then
\[
A(x,y)=(\{x-1,x, x+1\}\times \{y\})\;\cup \; (\{x\}\times \{y-1,y, y+1\}) \setminus\{(x,y)\}.
\]
Figure \ref{adyacencias_plano} illustrates the adjacencies in $\mathbb{K}^2$ (where the connecting segments represent adjacencies).

 \begin{figure}[H]
     \centering

\tikzset{every picture/.style={line width=0.75pt}} 

\begin{tikzpicture}[x=0.75pt,y=0.75pt,yscale=-1,xscale=1]

\draw [line width=1.5]    (311.24,91.1) -- (310.42,126.77) ;
\draw   (253.5,56.14) .. controls (253.5,51.99) and (257.09,48.63) .. (261.51,48.63) .. controls (265.93,48.63) and (269.51,51.99) .. (269.51,56.14) .. controls (269.51,60.29) and (265.93,63.66) .. (261.51,63.66) .. controls (257.09,63.66) and (253.5,60.29) .. (253.5,56.14) -- cycle ;
\draw  [fill={rgb, 255:red, 0; green, 0; blue, 0 }  ,fill opacity=1 ] (183.51,80.04) -- (186.59,87.77) -- (193.5,89.01) -- (188.5,95.03) -- (189.68,103.52) -- (183.51,99.51) -- (177.33,103.52) -- (178.51,95.03) -- (173.51,89.01) -- (180.42,87.77) -- cycle ;
\draw  [fill={rgb, 255:red, 0; green, 0; blue, 0 }  ,fill opacity=1 ] (222.51,41.04) -- (225.59,48.77) -- (232.5,50.01) -- (227.5,56.03) -- (228.68,64.52) -- (222.51,60.51) -- (216.33,64.52) -- (217.51,56.03) -- (212.51,50.01) -- (219.42,48.77) -- cycle ;
\draw  [fill={rgb, 255:red, 0; green, 0; blue, 0 }  ,fill opacity=1 ] (261.01,77.19) -- (264.09,84.92) -- (271,86.16) -- (266,92.18) -- (267.18,100.67) -- (261.01,96.66) -- (254.83,100.67) -- (256.01,92.18) -- (251.02,86.16) -- (257.92,84.92) -- cycle ;
\draw  [fill={rgb, 255:red, 0; green, 0; blue, 0 }  ,fill opacity=1 ] (224.51,120.04) -- (227.59,127.77) -- (234.5,129.01) -- (229.5,135.03) -- (230.68,143.52) -- (224.51,139.51) -- (218.33,143.52) -- (219.51,135.03) -- (214.51,129.01) -- (221.42,127.77) -- cycle ;
\draw [color={rgb, 255:red, 74; green, 144; blue, 226 }  ,draw opacity=1 ][line width=1.5]    (222.51,60.51) -- (223.01,84.18) ;
\draw [color={rgb, 255:red, 74; green, 144; blue, 226 }  ,draw opacity=1 ][line width=1.5]    (224,96.38) -- (224.51,120.04) ;
\draw [color={rgb, 255:red, 74; green, 144; blue, 226 }  ,draw opacity=1 ][line width=1.5]    (231.02,91.69) -- (256.01,92.18) ;
\draw [color={rgb, 255:red, 74; green, 144; blue, 226 }  ,draw opacity=1 ][line width=1.5]    (188.5,95.03) -- (213.5,95.51) ;
\draw [color={rgb, 255:red, 74; green, 144; blue, 226 }  ,draw opacity=1 ][line width=1.5]    (229.01,97.18) -- (255,130) ;
\draw [color={rgb, 255:red, 74; green, 144; blue, 226 }  ,draw opacity=1 ][line width=1.5]    (189.26,61.55) -- (216.01,85.18) ;
\draw [color={rgb, 255:red, 74; green, 144; blue, 226 }  ,draw opacity=1 ][line width=1.5]    (191.01,129) -- (217.01,99.18) ;
\draw [color={rgb, 255:red, 74; green, 144; blue, 226 }  ,draw opacity=1 ][line width=1.5]    (233.02,85.18) -- (257.26,60.8) ;
\draw [line width=1.5]    (189.01,135.03) -- (253.01,135.18) ;
\draw [line width=1.5]    (261.51,63.66) -- (260.51,129.66) ;
\draw [line width=1.5]    (183.01,62.18) -- (182.01,128.18) ;
\draw [line width=1.5]    (191.01,53.61) -- (254,54.43) ;
\draw   (345.32,44.32) -- (362.33,44.32) -- (362.33,56.35) -- (345.32,56.35) -- cycle ;
\draw  [fill={rgb, 255:red, 255; green, 255; blue, 255 }  ,fill opacity=1 ] (303.24,91.1) .. controls (303.24,86.95) and (306.82,83.58) .. (311.24,83.58) .. controls (315.66,83.58) and (319.25,86.95) .. (319.25,91.1) .. controls (319.25,95.25) and (315.66,98.62) .. (311.24,98.62) .. controls (306.82,98.62) and (303.24,95.25) .. (303.24,91.1) -- cycle ;
\draw   (348.99,126.69) -- (366,126.69) -- (366,138.72) -- (348.99,138.72) -- cycle ;
\draw   (385.33,90.84) .. controls (385.33,86.68) and (388.91,83.32) .. (393.33,83.32) .. controls (397.75,83.32) and (401.34,86.68) .. (401.34,90.84) .. controls (401.34,94.99) and (397.75,98.35) .. (393.33,98.35) .. controls (388.91,98.35) and (385.33,94.99) .. (385.33,90.84) -- cycle ;
\draw  [fill={rgb, 255:red, 0; green, 0; blue, 0 }  ,fill opacity=1 ] (312.33,40.38) -- (315.42,48.1) -- (322.32,49.34) -- (317.33,55.36) -- (318.51,63.85) -- (312.33,59.84) -- (306.16,63.85) -- (307.34,55.36) -- (302.34,49.34) -- (309.25,48.1) -- cycle ;
\draw  [fill={rgb, 255:red, 0; green, 0; blue, 0 }  ,fill opacity=1 ] (394.33,39.38) -- (397.42,47.1) -- (404.32,48.34) -- (399.33,54.36) -- (400.51,62.85) -- (394.33,58.84) -- (388.16,62.85) -- (389.34,54.36) -- (384.34,48.34) -- (391.25,47.1) -- cycle ;
\draw  [fill={rgb, 255:red, 0; green, 0; blue, 0 }  ,fill opacity=1 ] (354,75.7) -- (357.08,83.43) -- (363.99,84.67) -- (358.99,90.69) -- (360.17,99.18) -- (354,95.17) -- (347.82,99.18) -- (349,90.69) -- (344,84.67) -- (350.91,83.43) -- cycle ;
\draw  [fill={rgb, 255:red, 0; green, 0; blue, 0 }  ,fill opacity=1 ] (313.51,119.04) -- (316.59,126.77) -- (323.5,128.01) -- (318.5,134.03) -- (319.68,142.52) -- (313.51,138.51) -- (307.33,142.52) -- (308.51,134.03) -- (303.51,128.01) -- (310.42,126.77) -- cycle ;
\draw  [fill={rgb, 255:red, 0; green, 0; blue, 0 }  ,fill opacity=1 ] (392.51,121.05) -- (395.6,128.78) -- (402.5,130.02) -- (397.51,136.03) -- (398.69,144.53) -- (392.51,140.52) -- (386.33,144.53) -- (387.51,136.03) -- (382.52,130.02) -- (389.42,128.78) -- cycle ;
\draw [color={rgb, 255:red, 74; green, 144; blue, 226 }  ,draw opacity=1 ][line width=1.5]    (322.01,91.21) -- (349,90.69) ;
\draw [color={rgb, 255:red, 74; green, 144; blue, 226 }  ,draw opacity=1 ][line width=1.5]    (354,95.17) -- (354,121.69) ;
\draw [color={rgb, 255:red, 74; green, 144; blue, 226 }  ,draw opacity=1 ][line width=1.5]    (358.99,90.69) -- (384.33,89.84) ;
\draw [color={rgb, 255:red, 74; green, 144; blue, 226 }  ,draw opacity=1 ][line width=1.5]    (354,56.69) -- (354,75.7) ;
\draw [line width=1.5]    (367,133.69) -- (392.51,134.03) ;
\draw [line width=1.5]    (319.5,135.03) -- (350.01,134.37) ;
\draw [line width=1.5]    (361,50.69) -- (393.33,51.35) ;
\draw [line width=1.5]    (320,50.69) -- (345.51,51.03) ;
\draw [line width=1.5]    (394.33,52.35) -- (393.33,83.32) ;
\draw [line width=1.5]    (393.33,98.35) -- (392.51,134.03) ;
\draw [line width=1.5]    (312.33,53.35) -- (312.33,84.35) ;
\draw   (32,54) -- (49.01,54) -- (49.01,66.04) -- (32,66.04) -- cycle ;
\draw   (111,54) -- (128.01,54) -- (128.01,66.04) -- (111,66.04) -- cycle ;
\draw   (73.01,94.69) .. controls (73.01,90.54) and (76.59,87.18) .. (81.01,87.18) .. controls (85.43,87.18) and (89.02,90.54) .. (89.02,94.69) .. controls (89.02,98.85) and (85.43,102.21) .. (81.01,102.21) .. controls (76.59,102.21) and (73.01,98.85) .. (73.01,94.69) -- cycle ;
\draw   (32,132) -- (49.01,132) -- (49.01,144.04) -- (32,144.04) -- cycle ;
\draw   (113,133) -- (130.01,133) -- (130.01,145.04) -- (113,145.04) -- cycle ;
\draw  [fill={rgb, 255:red, 0; green, 0; blue, 0 }  ,fill opacity=1 ] (41.51,83.04) -- (44.59,90.77) -- (51.5,92.01) -- (46.5,98.03) -- (47.68,106.52) -- (41.51,102.51) -- (35.33,106.52) -- (36.51,98.03) -- (31.51,92.01) -- (38.42,90.77) -- cycle ;
\draw  [fill={rgb, 255:red, 0; green, 0; blue, 0 }  ,fill opacity=1 ] (80.51,44.04) -- (83.59,51.77) -- (90.5,53.01) -- (85.5,59.03) -- (86.68,67.52) -- (80.51,63.51) -- (74.33,67.52) -- (75.51,59.03) -- (70.51,53.01) -- (77.42,51.77) -- cycle ;
\draw  [fill={rgb, 255:red, 0; green, 0; blue, 0 }  ,fill opacity=1 ] (119.01,80.19) -- (122.09,87.92) -- (129,89.16) -- (124,95.18) -- (125.18,103.67) -- (119.01,99.66) -- (112.83,103.67) -- (114.01,95.18) -- (109.02,89.16) -- (115.92,87.92) -- cycle ;
\draw  [fill={rgb, 255:red, 0; green, 0; blue, 0 }  ,fill opacity=1 ] (82.51,123.04) -- (85.59,130.77) -- (92.5,132.01) -- (87.5,138.03) -- (88.68,146.52) -- (82.51,142.51) -- (76.33,146.52) -- (77.51,138.03) -- (72.51,132.01) -- (79.42,130.77) -- cycle ;
\draw [color={rgb, 255:red, 74; green, 144; blue, 226 }  ,draw opacity=1 ][line width=1.5]    (80.51,63.51) -- (81.01,87.18) ;
\draw [color={rgb, 255:red, 74; green, 144; blue, 226 }  ,draw opacity=1 ][line width=1.5]    (82,99.38) -- (82.51,123.04) ;
\draw [color={rgb, 255:red, 74; green, 144; blue, 226 }  ,draw opacity=1 ][line width=1.5]    (89.02,94.69) -- (114.01,95.18) ;
\draw [color={rgb, 255:red, 74; green, 144; blue, 226 }  ,draw opacity=1 ][line width=1.5]    (46.5,98.03) -- (71.5,98.51) ;
\draw [color={rgb, 255:red, 74; green, 144; blue, 226 }  ,draw opacity=1 ][line width=1.5]    (87.01,100.18) -- (113,133) ;
\draw [color={rgb, 255:red, 74; green, 144; blue, 226 }  ,draw opacity=1 ][line width=1.5]    (49.01,66.04) -- (74.01,88.18) ;
\draw [color={rgb, 255:red, 74; green, 144; blue, 226 }  ,draw opacity=1 ][line width=1.5]    (49.01,132) -- (75.01,102.18) ;
\draw [color={rgb, 255:red, 74; green, 144; blue, 226 }  ,draw opacity=1 ][line width=1.5]    (86.01,90) -- (111,66.04) ;
\draw [line width=1.5]    (47.01,138.03) -- (111.01,138.18) ;
\draw [line width=1.5]    (119.51,66.66) -- (118.51,132.66) ;
\draw [line width=1.5]    (41.01,65.18) -- (40.01,131.18) ;
\draw [line width=1.5]    (49.01,56.61) -- (112,57.43) ;
\draw   (216.01,85.18) -- (233.02,85.18) -- (233.02,97.21) -- (216.01,97.21) -- cycle ;
\draw   (252.5,137.18) .. controls (252.5,133.03) and (256.09,129.66) .. (260.51,129.66) .. controls (264.93,129.66) and (268.51,133.03) .. (268.51,137.18) .. controls (268.51,141.33) and (264.93,144.7) .. (260.51,144.7) .. controls (256.09,144.7) and (252.5,141.33) .. (252.5,137.18) -- cycle ;
\draw   (174.01,136.11) .. controls (174.01,132.18) and (177.81,129) .. (182.51,129) .. controls (187.2,129) and (191.01,132.18) .. (191.01,136.11) .. controls (191.01,140.03) and (187.2,143.21) .. (182.51,143.21) .. controls (177.81,143.21) and (174.01,140.03) .. (174.01,136.11) -- cycle ;
\draw   (175.01,54.66) .. controls (175.01,50.51) and (178.59,47.14) .. (183.01,47.14) .. controls (187.43,47.14) and (191.02,50.51) .. (191.02,54.66) .. controls (191.02,58.81) and (187.43,62.18) .. (183.01,62.18) .. controls (178.59,62.18) and (175.01,58.81) .. (175.01,54.66) -- cycle ;
\draw   (72.26,155.8) .. controls (72.29,160.47) and (74.63,162.78) .. (79.3,162.75) -- (144.3,162.36) .. controls (150.97,162.32) and (154.31,164.63) .. (154.34,169.3) .. controls (154.31,164.63) and (157.63,162.28) .. (164.3,162.23)(161.3,162.25) -- (229.3,161.84) .. controls (233.97,161.81) and (236.29,159.46) .. (236.26,154.8) ;
\draw   (306.26,157.8) .. controls (306.31,162.47) and (308.67,164.77) .. (313.34,164.71) -- (338.34,164.42) .. controls (345.01,164.34) and (348.37,166.63) .. (348.42,171.3) .. controls (348.37,166.63) and (351.67,164.26) .. (358.34,164.18)(355.34,164.21) -- (383.34,163.88) .. controls (388.01,163.83) and (390.31,161.47) .. (390.26,156.8) ;

\draw (108,173) node [anchor=north west][inner sep=0.75pt]   [align=left] {{\footnotesize Pure points}};
\draw (315,174) node [anchor=north west][inner sep=0.75pt]   [align=left] {{\footnotesize Mixed points}};

\end{tikzpicture}
\caption{Adjacency in the digital plane  $\mathbb{K}^2$.}
     \label{adyacencias_plano}
 \end{figure}
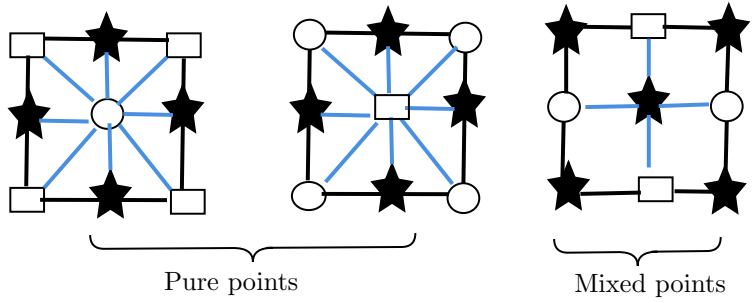
From Figure \ref{adyacencias_plano}, the following proposition is obtained.

\begin{propo} \label{dospuntosmixtos} 
Let $a, b \in \mathbb{K}^2$ be pure points. Then $a$ and $b$ share exactly two adjacent mixed points if and only if $a$ and $b$ are adjacent. 
\end{propo}

We now define paths and arcs, which play a fundamental role throughout this work.

\begin{defi}
Let $X$ be an Alexandroff space. A \textbf{path} (respectively, an \textbf{arc}) in $X$ is the continuous (respectively, homeomorphic) image of a finite interval of the digital line $\mathbb{K}$. 
\end{defi}

In Alexandroff spaces, being connected and being arc-connected are equivalent (see for instance \cite[Theorem 3.2]{Khalimsky} or  \cite[Proposition 3]{slapal2006}). In particular, the digital plane $\mathbb{K}^2$ is arc-connected. Following \cite{Khalimsky} we define the Jordan curves.

\begin{defi}
Let $X$ be an Alexandroff space. A \textbf{Jordan curve} in $X$ is a connected subset $J \subseteq X$ with $\left| J \right| \geq 4$ such that $J \setminus \{x\}$ is an arc for every $x\in J$. 
\end{defi}

There is only one Jordan curve with four points. We always work with Jordan curves with more than four points.

\begin{propo}
\cite[Lemma 5.2]{Khalimsky} 
\label{curvacarac}
Let $J$ be a finite subset of an Alexandroff space $X$. $J$ is a Jordan curve if and only if the following conditions hold: 

\begin{enumerate}
\item $J$ is connected. 

\item $J$ has at least four points. 

\item $|A(x) \cap J| = 2$ for each $x \in J$. 
\end{enumerate}

If $J$ is a Jordan curve, then $A(x) \cap (J \setminus \{x\})$ is disconnected. 
\end{propo}

One of the goals of this article is to explore an alternative proof of the topological version of the Jordan curve theorem. For that end, we will use an immersion of Rosenfeld's plane into the digital plane $\mathbb{K}^2$.
Consider the function (referred to as the \emph{slant map} in \cite{Khalimsky2}) $\Gamma: \mathbb{Z}^2 \to \mathbb{K}^2$ defined by:
$$
 \Gamma (x,y)=(x+y,y-x).
 $$
Figure \ref{sumerplano} illustrates the effect of $\Gamma$ on a subset of $\mathbb{Z}^2$. Notice that $\Gamma$ is injective and that the range of $\Gamma$ consists of pure points.

\begin{figure}[h]
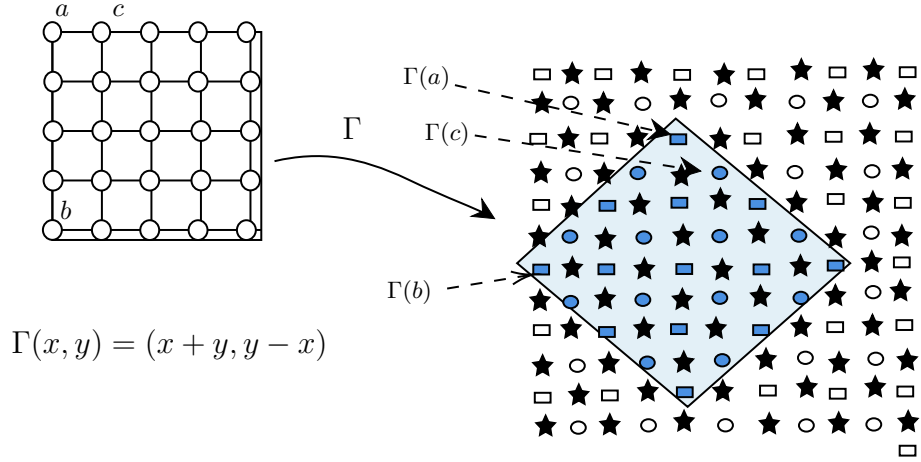

    \centering

\tikzset{every picture/.style={line width=0.75pt}} 


    \caption{Rosenfeld plane embedded in the digital plane.}
    \label{sumerplano}
\end{figure}

The following observations are taken from \cite[p. 51]{Khalimsky2}.

\begin{propo}
\label{4adyacentes}
Let $a, b \in \mathbb{Z}^2$.

\begin{enumerate}
\item $a$ and $b$ are 4-adjacent in $\mathbb{Z}^2$ if and only if $\Gamma(b) \in A(\Gamma(a))$.

\item $a$ and $b$ are 8-adjacent but not 4-adjacent in $\mathbb{Z}^2$ if and only if there exists a unique mixed point $c$ such that $\Gamma(b), \Gamma(a) \in A(c)$.
\end{enumerate}
\end{propo}

In Figure \ref{relacion_gama_z}, we illustrate  the proposition above, indicating the three points $a$, $b$, and $c$.

\begin{figure}[h]
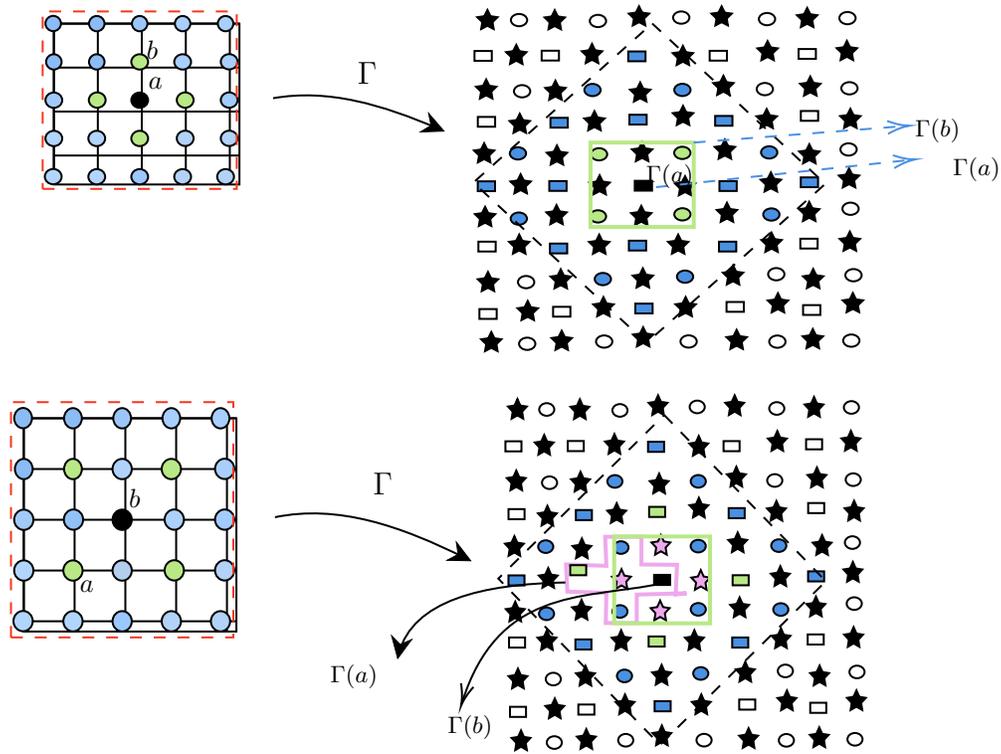

    \centering

\tikzset{every picture/.style={line width=0.75pt}} 



    \caption{$a$ and $b$ are  as in Proposition \ref{4adyacentes}.}
    \label{relacion_gama_z}
\end{figure}

We note that the use of coordinates to represent the points of $\mathbb{Z}^2$ is needed only in the proof of Proposition \ref{4adyacentes}, for the rest of what follows, we only use the fact that $\Gamma$ is one-to-one, its range consists of all pure points and it satisfies Proposition \ref{4adyacentes}.
For instance,  we could have used  $\Gamma_c(x,y)=(x+y,y-x+c)$  with $c$ any even integer.

\section{ Preservation of arcs and curves}
\label{sec-preser-curvas}

In this section, we begin the study of the relationship between connectivity in $\mathbb{Z}^2$ 
  (i.e., in the graph-theoretic sense) and connectivity in 
$\mathbb{K}^2$ (i.e., in the topological sense). We will analyze the behavior of connectivity under the functions 
 $\Gamma$, $\Gamma^{-1}$, and another function $\Gamma^*$ 
that we will define later. Our focus will be on the preservation of arcs in $\mathbb{K}^2$
 and curves in $\mathbb{Z}^2$, leaving the study of connectivity preservation for the subsequent sections.

The following proposition was stated without proof in \cite{Khalimsky2}; we include its proof for completeness.

\begin{propo}
\label{caminor1}
Let $C$ be an arc in $\mathbb{K}^2$ that consists only of pure points and with endpoints $z$ and $w$. Then $\Gamma^{-1}(C)$ is a 4-path with endpoints $\Gamma^{-1}(w)$ and $\Gamma^{-1}(z)$.
\end{propo}

\begin{proof}
Let $x \in \Gamma^{-1}(C)$ such that $x \notin \{ \Gamma^{-1}(w), \Gamma^{-1}(z) \}$, and let $v, y$ be such that $A(\Gamma(x)) \cap C = \{ v, y \}$. Since $y, v$ are pure points, by Proposition \ref{4adyacentes}, it follows that $\Gamma^{-1}(y)$ and $\Gamma^{-1}(v)$ are the two 4-adjacent points to $x$ in $\Gamma^{-1}(C)$. Similarly, if $x \in \{ \Gamma^{-1}(w), \Gamma^{-1}(z) \}$, then $x$ has only one 4-adjacent point in $\Gamma^{-1}(C)$. Therefore, $\Gamma^{-1}(C)$ is a 4-path.
\end{proof}

\begin{figure}[h]
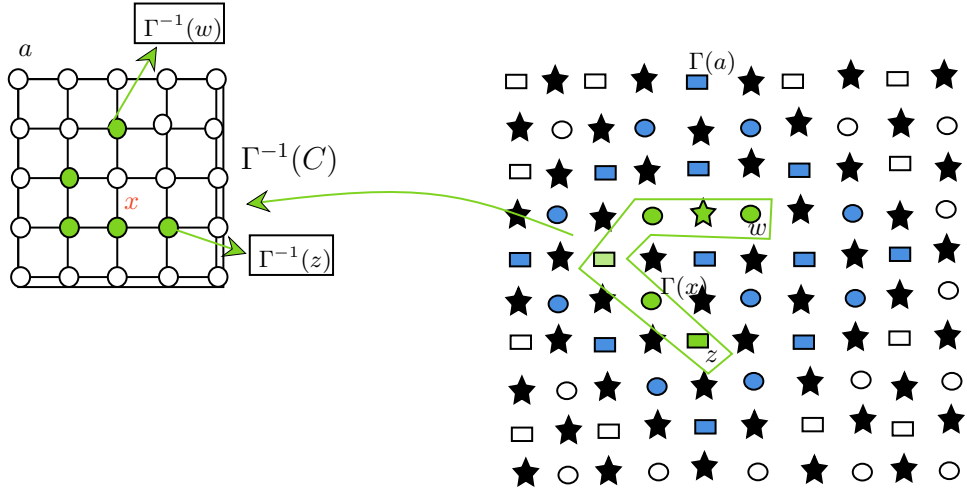

    \centering

\tikzset{every picture/.style={line width=0.75pt}} 


\caption{$C$ is an arc, but $\Gamma^{-1}(C)$ is not an  8-path. Note that  $\Gamma(x)$  has two adjacent points in  $C$, but $x$ has three  8-adjacent points in  $\Gamma^{-1}(C)$.}
    \label{caso1.1}
\end{figure}

The previous result does not hold in general, as illustrated in Figure \ref{caso1.1}. There exist arcs $C$ in $\mathbb{K}^2$ such that $\Gamma^{-1}(C)$ is not an 8-path. It is natural to ask what conditions are necessary (or sufficient) for $\Gamma^{-1}(C)$ to be an 8-path. To answer this question, we introduce a function $\Gamma^*$ that transforms subsets of $\mathbb{Z}^2$ into subsets of $\mathbb{K}^2$. For each $A \subseteq \mathbb{Z}^2$, we define
\begin{equation}
\label{defoperador*}
\begin{aligned}
&\Gamma^*(A)=\Gamma(A) \cup\{x\in \mathbb{K}^2: x \text {  is a mixed point such that } 
\\
& \qquad  \qquad \qquad 
N(x) \subseteq \Gamma(A) \cup\{x\} \text { or } cl(x) \subseteq \Gamma(A) \cup\{x\}\}.
\end{aligned}
\end{equation}

Note that $\Gamma^{-1}(\Gamma^*(B)) = B$ for every $B \subseteq \mathbb{Z}^2$. This implies that $\Gamma(A)$ is the set of pure points belonging to $\Gamma^*(A)$. The operator $\Gamma^*$ plays a crucial role in what follows and was motivated by a similar operator defined in \cite{Khalimsky2}.

A central problem is to determine the form of $\Gamma^{-1}(J)$ when $J$ is a Jordan curve in $\mathbb{K}^2$, in such a way that we can apply the Rosenfeld's Jordan curve theorem in $\mathbb{Z}^2$.

One of the main results of this section is that $C \subseteq \mathbb{Z}^2$ is an 8-path if and only if $\Gamma^{*}(C)$ is an arc in $\mathbb{K}^2$ (see Theorem \ref{camino*}). To prove this, we will need several auxiliary results, in particular,  we need to analyze the properties of $\Gamma^*(\Gamma^{-1}(J))$ when $J$ is a Jordan curve.

\begin{propo}
\label{1contenencia}
(i) For every $A \subseteq \mathbb{K}^2$, the pure points of $\Gamma^{*}(\Gamma^{-1} (A))$ are exactly the pure points of $A$. In particular, if all points of $A$ are pure, then $A \subseteq \Gamma^{*}(\Gamma^{-1} (A))$.

(ii) $J \subseteq \Gamma^{*}(\Gamma^{-1} (J))$ for every Jordan curve $J$ in $\mathbb{K}^2$.

(iii) If $C$ is an arc in $\mathbb{K}^2$, then
$C \setminus \{a \in C : a \text{ is a mixed endpoint}\} \subseteq \Gamma^{*}(\Gamma^{-1} (C))$. Moreover,   $x\not\in \Gamma^{*}(\Gamma^{-1} (C))$ if $x$ is a mixed final end point of  $C$.
\end{propo}

\begin{proof} 
(i) This is immediate from the definitions.

(ii) Suppose $J$ is a Jordan curve and $x \in J$. By (i), we may assume that $x$ is mixed. Since $J$ is a Jordan curve, we have that $A(x) \cap J = \{a, b\}$, where $a$ and $b$ are pure points. Since $J$ cannot turn at a mixed point, we conclude that either $cl(x) = \{a, b, x\}$ or $N(x) = \{a, b, x\}$. Thus, $x \in \Gamma^{*}(\Gamma^{-1} (J))$. 

(iii) This is analogous to (ii).
\end{proof}

\begin{propo}
\label{suficiente_8}
Let $C \subseteq \mathbb{K}^2$ such that 
$\Gamma^{*}(\Gamma^{-1}(C)) = C$. If $C$ is an arc in $\mathbb{K}^2$, then $\Gamma^{-1}(C)$ is an 8-path in $\mathbb{Z}^2$.
\end{propo}

\begin{proof}
Let $z$ and $w$ be the endpoints of $C$, by Proposition \ref{1contenencia}, $z$ and $w$ are pure points.  Let $x \in \Gamma^{-1}(C)$. We will consider two cases.

\medskip

\noindent\textbf{Case 1:} Suppose $x \notin \{\Gamma^{-1}(z), \Gamma^{-1}(w)\}$. We want to show that $x$ has exactly two 8-adjacent points in $\Gamma^{-1}(C)$. Since $w \neq \Gamma(x) \neq z$ and $C$ is an arc, there exist $u, v$ such that $A(\Gamma(x)) \cap C = \{u, v\}$. Consider the following subcases:

\medskip 

\begin{itemize}
\item[] \hspace{-1,2cm}\textbf{1.1:} Suppose that $u$ and $v$ are pure points. By Proposition \ref{4adyacentes}, $\Gamma^{-1}(u)$ and $\Gamma^{-1}(v)$ are 4-adjacent to $x$, and in particular, they are 8-adjacent to $x$. Now, we will show that $x$ has no other 8-adjacent points in $\Gamma^{-1}(C)$. Suppose not, and let $y \in \Gamma^{-1}(C)$ be another 8-adjacent point to $x$, such that $\Gamma^{-1}(u) \neq y \neq \Gamma^{-1}(v)$. We claim that $y$ is not 4-adjacent to $x$. Indeed, suppose that $y$ is 4-adjacent to $x$; then by Proposition \ref{4adyacentes}, $\Gamma(y)$ must be adjacent to $\Gamma(x)$. Since $\Gamma$ is injective, $u \neq \Gamma(y) \neq v$, so $|A(\Gamma(x)) \cap C| = 3$, which contradicts the fact that $C$ is an arc. Therefore, $x$ and $y$ are 8-adjacent but not 4-adjacent, and by Proposition \ref{4adyacentes}, there exists a unique mixed point $w_1$ such that $\Gamma(x), \Gamma(y) \in A(w_1)$. Since $\Gamma(x)$ and $\Gamma(y)$ are not adjacent, we have either $N(w_1) = \{\Gamma(x), \Gamma(y), w_1\}$ or $cl(w_1) = \{\Gamma(x), \Gamma(y), w_1\}$. Thus, $w_1 \in \Gamma^{*}(\Gamma^{-1}(C)) = C$, which implies that $w_1 \in A(\Gamma(x)) \cap C$. Since $u$ and $v$ are pure points, we have $|A(\Gamma(x)) \cap C| = 3$, which contradicts the fact that $C$ is an arc.
This shows that $x$ has only two 8-adjacent points in $\Gamma^{-1}(C)$.

\medskip

\item[]\hspace{-1,2cm} \textbf{1.2:}  Suppose that $u$ and $v$ are mixed points. Then neither $u$ nor $v$ are endpoints of $C$, so $|A(u) \cap C| = |A(v) \cap C| = 2$. It follows that $A(u) \cap C = \{\Gamma(x), y\}$ and $A(v) \cap C = \{\Gamma(x), z_1\}$, where $y$ and $z_1$ are pure points of $C$. 

Note that $\Gamma(x)$ and $y$ are not adjacent, since $A(\Gamma(x)) \cap C = \{u, v\}$. Thus, $u$ is the only mixed point that is simultaneously adjacent  to both $\Gamma(x)$ and $y$ (see Figure \ref{adyacencias_plano}). Similarly, $v$ is the only mixed point adjacent to both $z_1$ and $\Gamma(x)$. Therefore, by Proposition \ref{4adyacentes}, $\Gamma^{-1}(z_1)$ and $\Gamma^{-1}(y)$ are 8-adjacent to $x$. Thus, $x$ has at least two 8-adjacent points in $\Gamma^{-1}(C)$. 

Suppose that $w_1 \in \Gamma^{-1}(C)$ is another 8-adjacent point to $x$, distinct from $\Gamma^{-1}(z_1)$ and $\Gamma^{-1}(y)$. Observe that $w_1$ is not 4-adjacent to $x$, since $|A(\Gamma(x)) \cap C| = 2$. By Proposition \ref{4adyacentes}, there exists a unique mixed point $w'_1$ adjacent to both $\Gamma(x)$ and $\Gamma(w_1)$. Using the same reasoning as in case 1.1, we conclude that $w'_1 \in C$. 

Now, let us show that $u \neq w'_1 \neq v$. Suppose that $w'_1 = u$. Then we would have $y$, $\Gamma(x)$, and $\Gamma(w_1) \in A(u)$. Since $\Gamma(x) \neq \Gamma(w_1) \neq y$ and $y$ is not adjacent to $\Gamma(x)$, we see that since $u$ is a mixed point,  $\Gamma(w_1)$ must be adjacent to $\Gamma(x)$ (remember the adjacency of a mixed point, seen figure \ref{adyacencias_plano}), which cannot be true because $|A(\Gamma(x)) \cap C| = 2$ and $u \neq \Gamma(w_1) \neq v$. Therefore, $u \neq w'_1$. A similar argument shows that $v \neq w'_1$. Thus, $|A(\Gamma(x)) \cap C| = 3$, which contradicts the fact that $C$ is an arc. This shows that $x$ has only two 8-adjacent points in $\Gamma^{-1}(C)$.

\medskip

\item[]\hspace{-1,2cm} \textbf{1.3:} Suppose $u$ is mixed and $v$ is pure. Using similar reasoning as in the previous cases, we conclude that $x$ has exactly two 8-adjacent points in $\Gamma^{-1}(C)$.
\end{itemize}

\medskip 

\noindent \textbf{Case 2:} Suppose $x \in \{\Gamma^{-1}(z), \Gamma^{-1}(w)\}$. Since $\Gamma(x)$ is an endpoint of $C$, we have $A(\Gamma(x)) \cap C = \{u\}$. If $u$ is pure, then $\Gamma^{-1}(u)$ is 8-adjacent to $x$ in $\Gamma^{-1}(C)$. If $u$ is mixed, then $A(u) \cap C = \{x, z_1\}$, where $z_1$ is a pure point. By Proposition \ref{4adyacentes}, $\Gamma^{-1}(z_1)$ is 8-adjacent to $x$, since $z_1$ is not adjacent to $\Gamma(x)$ and has a common adjacent mixed point. 

We have shown that $x$ has at least one 8-adjacent point in $\Gamma^{-1}(C)$. Using similar reasoning as in Case 1 and noting that $z$ and $w$ are endpoints of $C$, we conclude that $x$ has exactly one 8-adjacent point in $\Gamma^{-1}(C)$. In other words, $\Gamma^{-1}(z)$ and $\Gamma^{-1}(w)$ are the endpoints of $\Gamma^{-1}(C)$.

This concludes the proof that $\Gamma^{-1}(C)$ is an 8-path.
\end{proof}

As a corollary of the previous arguments, we now describe the conditions under which the preimage of a Jordan curve  under $\Gamma$ is a  simple curve in the Rosenfeld plane $\mathbb{Z}^2$.

\begin{propo}
\label{4-curva} 
\label{8-curva}
Let $J$ be a Jordan curve in $\mathbb{K}^2$. 

\begin{itemize} 
\item[(i)] If $J$ contains only pure points, then $\Gamma^{-1}(J)$ is a closed  4-curve in $\mathbb{Z}^2$. 

\item[(ii)] If $J$ satisfies $\Gamma^{*}(\Gamma^{-1}(J)) = J$, then $\Gamma^{-1}(J)$ is a closed 8-curve in $\mathbb{Z}^2$. 
\end{itemize} 
\end{propo} 

\begin{proof} \text{(i)} Let $e \in \Gamma^{-1}(J)$. We show that $e$ has exactly two 4-adjacent points in $\Gamma^{-1}(J)$. Since $J$ is a Jordan curve and $\Gamma(e) \in J$, it follows that $\lvert A(e) \cap J \rvert = 2$. By Proposition \ref{4adyacentes}, we conclude that $e$ has exactly two 4-adjacent points in $\Gamma^{-1}(J)$. 

\text{(ii)} To prove that $\Gamma^{-1}(J)$ is an 8-curve, it suffices to show that every point $x \in \Gamma^{-1}(J)$ has exactly two 8-adjacent points. Since $\Gamma(x) \in J$ and $J$ is a Jordan curve, we have $\lvert A(\Gamma(x)) \cap J \rvert = 2$. Following an analogous reasoning to Proposition \ref{suficiente_8}, case 1, we deduce that $x$ has exactly two 8-adjacent points in $\Gamma^{-1}(J)$. Therefore, $\Gamma^{-1}(J)$ is an 8-curve. 

\end{proof}

We now analyze  $\Gamma(C)$ when  $C$ is  an 8-path in $\mathbb{Z}^2$.

\begin{figure}[h]
    \centering
\tikzset{every picture/.style={line width=0.75pt}} 

\begin{tikzpicture}[x=0.75pt,y=0.75pt,yscale=-1,xscale=1]

\draw  [fill={rgb, 255:red, 255; green, 255; blue, 255 }  ,fill opacity=1 ] (57.8,60.53) -- (74.3,60.53) -- (74.3,72.19) -- (57.8,72.19) -- cycle ;
\draw  [fill={rgb, 255:red, 255; green, 255; blue, 255 }  ,fill opacity=1 ] (138.06,61.9) -- (154.55,61.9) -- (154.55,73.57) -- (138.06,73.57) -- cycle ;
\draw  [fill={rgb, 255:red, 255; green, 255; blue, 255 }  ,fill opacity=1 ] (61.56,133.28) -- (78.06,133.28) -- (78.06,144.95) -- (61.56,144.95) -- cycle ;
\draw  [color={rgb, 255:red, 0; green, 0; blue, 0 }  ,draw opacity=1 ][fill={rgb, 255:red, 255; green, 255; blue, 255 }  ,fill opacity=1 ][line width=0.75]  (98.52,107.67) .. controls (98.52,103.64) and (102,100.38) .. (106.28,100.38) .. controls (110.57,100.38) and (114.04,103.64) .. (114.04,107.67) .. controls (114.04,111.69) and (110.57,114.96) .. (106.28,114.96) .. controls (102,114.96) and (98.52,111.69) .. (98.52,107.67) -- cycle ;
\draw  [fill={rgb, 255:red, 255; green, 255; blue, 255 }  ,fill opacity=1 ] (138.29,133) -- (154.78,133) -- (154.78,144.67) -- (138.29,144.67) -- cycle ;
\draw  [color={rgb, 255:red, 0; green, 0; blue, 0 }  ,draw opacity=1 ][fill={rgb, 255:red, 255; green, 255; blue, 255 }  ,fill opacity=1 ][line width=0.75]  (174.21,106.46) .. controls (174.21,102.44) and (177.69,99.18) .. (181.97,99.18) .. controls (186.26,99.18) and (189.74,102.44) .. (189.74,106.46) .. controls (189.74,110.49) and (186.26,113.75) .. (181.97,113.75) .. controls (177.69,113.75) and (174.21,110.49) .. (174.21,106.46) -- cycle ;
\draw  [fill={rgb, 255:red, 0; green, 0; blue, 0 }  ,fill opacity=1 ] (108.83,54.11) -- (111.83,61.6) -- (118.52,62.8) -- (113.68,68.63) -- (114.82,76.87) -- (108.83,72.98) -- (102.85,76.87) -- (103.99,68.63) -- (99.15,62.8) -- (105.84,61.6) -- cycle ;
\draw  [color={rgb, 255:red, 0; green, 0; blue, 0 }  ,draw opacity=1 ][fill={rgb, 255:red, 0; green, 0; blue, 0 }  ,fill opacity=1 ] (179.54,54.85) -- (182.53,62.34) -- (189.23,63.54) -- (184.38,69.37) -- (185.53,77.61) -- (179.54,73.72) -- (173.55,77.61) -- (174.7,69.37) -- (169.85,63.54) -- (176.55,62.34) -- cycle ;
\draw  [fill={rgb, 255:red, 0; green, 0; blue, 0 }  ,fill opacity=1 ] (67.44,93.64) -- (70.43,101.13) -- (77.13,102.33) -- (72.28,108.16) -- (73.43,116.4) -- (67.44,112.51) -- (61.45,116.4) -- (62.6,108.16) -- (57.75,102.33) -- (64.45,101.13) -- cycle ;
\draw  [fill={rgb, 255:red, 0; green, 0; blue, 0 }  ,fill opacity=1 ] (143.98,92.67) -- (146.97,100.16) -- (153.67,101.36) -- (148.82,107.19) -- (149.97,115.43) -- (143.98,111.54) -- (137.99,115.43) -- (139.14,107.19) -- (134.29,101.36) -- (140.99,100.16) -- cycle ;
\draw  [fill={rgb, 255:red, 0; green, 0; blue, 0 }  ,fill opacity=1 ] (102.48,125.45) -- (105.47,132.94) -- (112.16,134.14) -- (107.32,139.97) -- (108.46,148.21) -- (102.48,144.32) -- (96.49,148.21) -- (97.63,139.97) -- (92.79,134.14) -- (99.48,132.94) -- cycle ;
\draw  [fill={rgb, 255:red, 0; green, 0; blue, 0 }  ,fill opacity=1 ] (176.96,124.41) -- (179.96,131.9) -- (186.65,133.11) -- (181.81,138.94) -- (182.95,147.17) -- (176.96,143.28) -- (170.98,147.17) -- (172.12,138.94) -- (167.28,133.11) -- (173.97,131.9) -- cycle ;
\draw  [color={rgb, 255:red, 74; green, 144; blue, 226 }  ,draw opacity=1 ][dash pattern={on 1.69pt off 2.76pt}][line width=1.5]  (95.98,98.35) .. controls (95.98,95.81) and (98.04,93.75) .. (100.58,93.75) -- (187.38,93.75) .. controls (189.92,93.75) and (191.98,95.81) .. (191.98,98.35) -- (191.98,112.15) .. controls (191.98,114.69) and (189.92,116.75) .. (187.38,116.75) -- (100.58,116.75) .. controls (98.04,116.75) and (95.98,114.69) .. (95.98,112.15) -- cycle ;
\draw  [color={rgb, 255:red, 126; green, 211; blue, 33 }  ,draw opacity=1 ][dash pattern={on 1.69pt off 2.76pt}][line width=1.5]  (154.8,57.53) .. controls (157.78,57.55) and (160.17,59.99) .. (160.15,62.97) -- (159.47,149.81) .. controls (159.45,152.79) and (157.01,155.19) .. (154.03,155.16) -- (137.85,155.04) .. controls (134.87,155.01) and (132.47,152.58) .. (132.49,149.6) -- (133.18,62.76) .. controls (133.2,59.78) and (135.63,57.38) .. (138.61,57.4) -- cycle ;
\draw  [fill={rgb, 255:red, 255; green, 255; blue, 255 }  ,fill opacity=1 ] (356.64,39.19) -- (373.14,39.19) -- (373.14,50.86) -- (356.64,50.86) -- cycle ;
\draw  [fill={rgb, 255:red, 255; green, 255; blue, 255 }  ,fill opacity=1 ] (355.06,108.9) -- (371.55,108.9) -- (371.55,120.57) -- (355.06,120.57) -- cycle ;
\draw  [color={rgb, 255:red, 0; green, 0; blue, 0 }  ,draw opacity=1 ][fill={rgb, 255:red, 255; green, 255; blue, 255 }  ,fill opacity=1 ][line width=0.75]  (320.52,76.67) .. controls (320.52,72.64) and (324,69.38) .. (328.28,69.38) .. controls (332.57,69.38) and (336.04,72.64) .. (336.04,76.67) .. controls (336.04,80.69) and (332.57,83.96) .. (328.28,83.96) .. controls (324,83.96) and (320.52,80.69) .. (320.52,76.67) -- cycle ;
\draw  [color={rgb, 255:red, 0; green, 0; blue, 0 }  ,draw opacity=1 ][fill={rgb, 255:red, 255; green, 255; blue, 255 }  ,fill opacity=1 ][line width=0.75]  (395.21,145.46) .. controls (395.21,141.44) and (398.69,138.18) .. (402.97,138.18) .. controls (407.26,138.18) and (410.74,141.44) .. (410.74,145.46) .. controls (410.74,149.49) and (407.26,152.75) .. (402.97,152.75) .. controls (398.69,152.75) and (395.21,149.49) .. (395.21,145.46) -- cycle ;
\draw  [fill={rgb, 255:red, 0; green, 0; blue, 0 }  ,fill opacity=1 ] (363.83,64.4) -- (366.83,71.89) -- (373.52,73.09) -- (368.68,78.92) -- (369.82,87.16) -- (363.83,83.27) -- (357.85,87.16) -- (358.99,78.92) -- (354.15,73.09) -- (360.84,71.89) -- cycle ;
\draw  [color={rgb, 255:red, 0; green, 0; blue, 0 }  ,draw opacity=1 ][fill={rgb, 255:red, 0; green, 0; blue, 0 }  ,fill opacity=1 ] (395.54,30.85) -- (398.53,38.34) -- (405.23,39.54) -- (400.38,45.37) -- (401.53,53.61) -- (395.54,49.72) -- (389.55,53.61) -- (390.7,45.37) -- (385.85,39.54) -- (392.55,38.34) -- cycle ;
\draw  [fill={rgb, 255:red, 0; green, 0; blue, 0 }  ,fill opacity=1 ] (324.44,101.64) -- (327.43,109.13) -- (334.13,110.33) -- (329.28,116.16) -- (330.43,124.4) -- (324.44,120.51) -- (318.45,124.4) -- (319.6,116.16) -- (314.75,110.33) -- (321.45,109.13) -- cycle ;
\draw  [fill={rgb, 255:red, 0; green, 0; blue, 0 }  ,fill opacity=1 ] (400.98,100.67) -- (403.97,108.16) -- (410.67,109.36) -- (405.82,115.19) -- (406.97,123.43) -- (400.98,119.54) -- (394.99,123.43) -- (396.14,115.19) -- (391.29,109.36) -- (397.99,108.16) -- cycle ;
\draw  [fill={rgb, 255:red, 0; green, 0; blue, 0 }  ,fill opacity=1 ] (365.48,129.45) -- (368.47,136.94) -- (375.16,138.14) -- (370.32,143.97) -- (371.46,152.21) -- (365.48,148.32) -- (359.49,152.21) -- (360.63,143.97) -- (355.79,138.14) -- (362.48,136.94) -- cycle ;
\draw  [fill={rgb, 255:red, 0; green, 0; blue, 0 }  ,fill opacity=1 ] (325.96,30.41) -- (328.96,37.9) -- (335.65,39.11) -- (330.81,44.94) -- (331.95,53.17) -- (325.96,49.28) -- (319.98,53.17) -- (321.12,44.94) -- (316.28,39.11) -- (322.97,37.9) -- cycle ;
\draw  [color={rgb, 255:red, 74; green, 144; blue, 226 }  ,draw opacity=1 ][dash pattern={on 1.69pt off 2.76pt}][line width=1.5]  (315.83,70.08) .. controls (315.83,67.54) and (317.89,65.48) .. (320.43,65.48) -- (407.23,65.48) .. controls (409.78,65.48) and (411.83,67.54) .. (411.83,70.08) -- (411.83,83.88) .. controls (411.83,86.42) and (409.78,88.48) .. (407.23,88.48) -- (320.43,88.48) .. controls (317.89,88.48) and (315.83,86.42) .. (315.83,83.88) -- cycle ;
\draw  [color={rgb, 255:red, 126; green, 211; blue, 33 }  ,draw opacity=1 ][dash pattern={on 1.69pt off 2.76pt}][line width=1.5]  (372.31,28.23) .. controls (375.29,28.25) and (377.69,30.69) .. (377.66,33.66) -- (376.98,120.51) .. controls (376.96,123.49) and (374.52,125.88) .. (371.54,125.86) -- (355.36,125.73) .. controls (352.38,125.71) and (349.98,123.27) .. (350.01,120.29) -- (350.69,33.45) .. controls (350.71,30.47) and (353.15,28.08) .. (356.13,28.1) -- cycle ;
\draw  [color={rgb, 255:red, 0; green, 0; blue, 0 }  ,draw opacity=1 ][fill={rgb, 255:red, 255; green, 255; blue, 255 }  ,fill opacity=1 ][line width=0.75]  (389.21,77.46) .. controls (389.21,73.44) and (392.69,70.18) .. (396.97,70.18) .. controls (401.26,70.18) and (404.74,73.44) .. (404.74,77.46) .. controls (404.74,81.49) and (401.26,84.75) .. (396.97,84.75) .. controls (392.69,84.75) and (389.21,81.49) .. (389.21,77.46) -- cycle ;
\draw  [color={rgb, 255:red, 0; green, 0; blue, 0 }  ,draw opacity=1 ][fill={rgb, 255:red, 255; green, 255; blue, 255 }  ,fill opacity=1 ][line width=0.75]  (317.21,142.46) .. controls (317.21,138.44) and (320.69,135.18) .. (324.97,135.18) .. controls (329.26,135.18) and (332.74,138.44) .. (332.74,142.46) .. controls (332.74,146.49) and (329.26,149.75) .. (324.97,149.75) .. controls (320.69,149.75) and (317.21,146.49) .. (317.21,142.46) -- cycle ;
\draw [color={rgb, 255:red, 126; green, 211; blue, 33 }  ,draw opacity=1 ][line width=2.25]  [dash pattern={on 2.53pt off 3.02pt}]  (447,81) -- (464.19,80.48) ;
\draw [color={rgb, 255:red, 74; green, 144; blue, 226 }  ,draw opacity=1 ][line width=2.25]  [dash pattern={on 2.53pt off 3.02pt}]  (449,98) -- (466.19,97.48) ;

\draw (77.77,83.8) node [anchor=north west][inner sep=0.75pt]  [font=\scriptsize]  {$\Gamma ( x_{0})$};
\draw (26.79,44.02) node [anchor=north west][inner sep=0.75pt]  [font=\scriptsize]  {$\Gamma ( x_{1} )$};
\draw (140.06,76.97) node [anchor=north west][inner sep=0.75pt]  [font=\scriptsize]  {$u$};
\draw (189.38,97.15) node [anchor=north west][inner sep=0.75pt]  [font=\scriptsize]  {$u'$};
\draw (334.77,91.8) node [anchor=north west][inner sep=0.75pt]  [font=\scriptsize]  {$\Gamma ( x_{0})$};
\draw (283.79,63.02) node [anchor=north west][inner sep=0.75pt]  [font=\scriptsize]  {$\Gamma ( x_{1} )$};
\draw (482,76.4) node [anchor=north west][inner sep=0.75pt]  [font=\scriptsize]  {$cl( u)$};
\draw (483,94.4) node [anchor=north west][inner sep=0.75pt]  [font=\scriptsize]  {$N( u)$};
\draw (403.83,62.48) node [anchor=north west][inner sep=0.75pt]  [font=\scriptsize]  {$u'$};
\draw (118,173) node [anchor=north west][inner sep=0.75pt]  [font=\large] [align=left] {A};
\draw (362,174) node [anchor=north west][inner sep=0.75pt]  [font=\large] [align=left] {B};

\end{tikzpicture}
    \caption{Case 1 in the proof of Proposition \ref{puntosfinales}}
    \label{(u,v)mixto}
\end{figure}
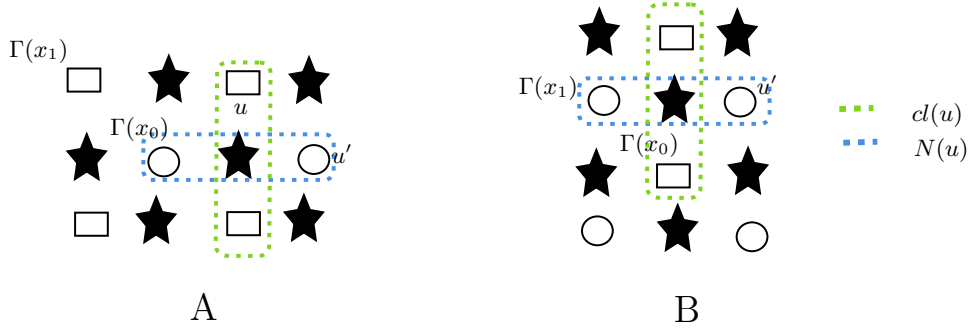

\begin{propo}
\label{puntosfinales} 
Let $C$ be an 8-path and $x$ an 8-endpoint of $C$. Then $\lvert A(\Gamma(x)) \cap \Gamma^{*}(C) \rvert = 1$. \end{propo} \begin{proof} Let $C = \{x_0, x_1, \ldots, x_{n-1}, x_n\}$ be an 8-path with 8-endpoints $x_0$ and $x_n$. We analyze the situation for $x_0$; the other endpoint is treated similarly. We consider two cases: 

\medskip

\noindent\textbf{Case 1:} Suppose $x_1$ is 4-adjacent to $x_0$. It follows that $\Gamma(x_1) \in A(\Gamma(x_0)) \cap \Gamma^{*}(C)$ (see Proposition \ref{4adyacentes}). Assume, by contradiction, that $\lvert A(\Gamma(x_0)) \cap \Gamma^{*}(C) \rvert > 1$ and let $u \in A(\Gamma(x_0)) \cap \Gamma^{*}(C)$ such that $u \neq \Gamma(x_1)$. We analyze two subcases: 

\medskip 

\begin{itemize}
\item[]\hspace{-1.2cm}\textbf{1.1:} If $u$ is pure, then $\Gamma^{-1}(u) \in C$ and $u$ is 8-adjacent to $x_0$ (see Proposition \ref{4adyacentes}). Since $\Gamma^{-1}(u) \neq x_1$, $x_0$ has two 8-adjacent points, contradicting the assumption that $x_0$ is an 8-endpoint of $C$. 

\item[]\hspace{-1.2cm}\textbf{1.2:} If $u$ is mixed, then $N(u) \subseteq \Gamma(C) \cup \{u\}$ or $cl(u) \subseteq \Gamma(C) \cup \{u\}$ (see Figure \ref{(u,v)mixto}). We assume $N(u) \subseteq \Gamma(C) \cup \{u\}$; the other case is treated similarly. There exists a pure point $u'$ in $N(u)$ such that $u' \neq \Gamma(x_1)$ or $u' \neq \Gamma(x_0)$, with $u' \in \Gamma(C)$. As $\Gamma(x_0)$ and $u'$ are not adjacent and two pure points share at most two mixed adjacent points, Proposition \ref{dospuntosmixtos} implies $u$ is the only mixed point adjacent to $\Gamma(x_0)$ and $u'$. Hence, $\Gamma^{-1}(u')$ is 8-adjacent to $x_0$, contradicting that $x_0$ is an 8-endpoint. \end{itemize} 

Thus, $\lvert A(\Gamma(x_0)) \cap \Gamma^{*}(C) \rvert = 1$ in this case. 

\medskip 

\noindent\textbf{Case 2:} Suppose $x_1$ is 8-adjacent but not 4-adjacent to $x_0$. By Proposition \ref{4adyacentes}, there exists a unique mixed point $u$ such that $\Gamma(x_1), \Gamma(x_0) \in A(u)$. Since $\Gamma(x_1)$ and $\Gamma(x_0)$ are not adjacent, $N(u) = \{\Gamma(x_1), \Gamma(x_0), u\}$ or $cl(u) = \{\Gamma(x_1), \Gamma(x_0), u\}$, implying $u \in \Gamma^{*}(C)$. Following reasoning analogous to Case 1, it can be verified that $\lvert A(\Gamma(x_0)) \cap \Gamma^{*}(C) \rvert = 1$. 

\end{proof}
 
\begin{propo}
\label{puntosnofinales}
Let $C$ be an 8-path with 8-endpoints $x_0$ and $x_n$. If $z \in \Gamma^{*}(C)$ is such that $z \neq \Gamma(x_0)$ and $z \neq \Gamma(x_n)$, then $|A(z) \cap \Gamma^{*}(C)| = 2$.
\end{propo}

\begin{proof}
Let $C = \{x_0, x_1, \ldots, x_{n-1}, x_n\}$ be an 8-path with 8-endpoints $x_0$ and $x_n$. We consider two cases.

\medskip

\noindent\textbf{Case 1:} Assume $z$ is mixed. Then, either $cl(z) \subseteq \Gamma(C) \cup \{z\}$ or $N(z) \subseteq \Gamma(C) \cup \{z\}$. Thus, there exists $i$ such that $\{\Gamma(x_i), \Gamma(x_{i+1})\} \subseteq N(z)$ or $\{\Gamma(x_i), \Gamma(x_{i+1})\} \subseteq cl(z)$. Since $\Gamma(C) \subseteq \Gamma^{*}(C)$, it follows that $|A(z) \cap \Gamma^{*}(C)| \geq 2$. Let $u \in A(z) \cap \Gamma^{*}(C)$. Because $z$ is mixed, $u$ is pure, which implies $u \in \Gamma(C)$. Let $0 \leq j \leq n$ be such that $u = \Gamma(x_j)$. We will show that $j = i$ or $j = i + 1$. Suppose otherwise. Various alternatives need to be considered, and in each, $x_j$ would have more 8-adjacent points than allowed, contradicting the fact that $C$ is an 8-path. Since $\Gamma(x_j) \in A(\Gamma(x_i)) \cap A(\Gamma(x_{i+1}))$, $x_j$ must be 4-adjacent to both $x_i$ and $x_{i+1}$ (by Proposition \ref{4adyacentes}). If $0 < j < i$, then $x_{j-1}$, $x_i$, and $x_{i+1}$ would all be 8-adjacent to $x_j$, which is impossible. The remaining cases are handled analogously.

\medskip

\noindent\textbf{Case 2:} Assume $z$ is a pure point. Then $\Gamma(x_i) = z$ for some $i \neq 0$ and $i \neq n$. Since $C$ is an 8-path, $x_{i-1}$ and $x_{i+1}$ are 8-adjacent to $x_i$. We analyze the following subcases:

\medskip 

\begin{itemize}
\item[] \hspace{-1.2cm}\textbf{2.1:} Suppose $x_i$ is 4-adjacent to both $x_{i-1}$ and $x_{i+1}$. By Proposition \ref{4adyacentes}, $\Gamma(x_{i-1}), \Gamma(x_{i+1}) \in A(z)$, so $|A(z) \cap \Gamma^{*}(C)| \geq 2$. Assume there exists $y \in A(z) \cap \Gamma^{*}(C)$ with $y \notin \{\Gamma(x_{i-1}), \Gamma(x_{i+1})\}$. Several cases have to be considered, and in all of them, $x_i$ would have more 8-adjacent points than allowed for an 8-path. If $y$ is pure, then $y = \Gamma(x_j)$, and $x_j$ would be 4-adjacent to $x_i$ by Proposition \ref{4adyacentes}, giving $x_i$ at least 3 8-adjacent points, which is impossible. If $y$ is mixed, then either $cl(y) \subseteq \Gamma(C) \cup \{y\}$ or $N(y) \subseteq \Gamma(C) \cup \{y\}$. In both cases, a pure point $\Gamma(x_j) \in A(z)$ would exist with $j \notin \{i-1, i+1\}$, which is not possible.

\medskip

\item[]\hspace{-1.2cm}\textbf{2.2:} Suppose $x_i$ is 8-adjacent to both $x_{i-1}$ and $x_{i+1}$ but not 4-adjacent to either. By Proposition \ref{4adyacentes}, there exist mixed points $u$ and $u'$ such that $\Gamma(x_i), \Gamma(x_{i+1}) \in A(u)$ and $\Gamma(x_{i-1}), \Gamma(x_i) \in A(u')$. Since $x_i$ is not 4-adjacent to either $x_{i-1}$ or $x_{i+1}$, we have:
\begin{itemize}
\item[] $cl(u) = \{\Gamma(x_i), \Gamma(x_{i+1}), u\}$ or $N(u) = \{\Gamma(x_i), \Gamma(x_{i+1}), u\}$.
\item[] $cl(u') = \{\Gamma(x_i), \Gamma(x_{i-1}), u'\}$ or $N(u') = \{\Gamma(x_i), \Gamma(x_{i-1}), u'\}$.
\end{itemize}
Thus, $u, u' \in \Gamma^{*}(C)$, so $|A(z) \cap \Gamma^{*}(C)| \geq 2$. If $|A(z) \cap \Gamma^{*}(C)| > 2$, Proposition \ref{4adyacentes} implies that $x_i$ would have three  8-adjacent points in $C$, contradicting the assumption that $C$ is an 8-path. Therefore, $|A(z) \cap \Gamma^{*}(C)| = 2$.

\medskip

\item[]\hspace{-1.2cm}\textbf{2.3:} Suppose $x_i$ is not 4-adjacent to $x_{i-1}$ but is 4-adjacent to $x_{i+1}$. Recall that $x_i$ is 8-adjacent to both $x_{i-1}$ and $x_{i+1}$. Using the same reasoning as in the previous case, there exists a mixed point $u$ such that $u \in A(z) \cap \Gamma^{*}(C)$, and $\Gamma(x_{i-1}) \in A(z) \cap \Gamma^{*}(C)$. Thus, $|A(z) \cap \Gamma^{*}(C)| \geq 2$. Following analogous reasoning as in earlier cases, we conclude that $|A(z) \cap \Gamma^{*}(C)| = 2$.
\end{itemize}
\end{proof}

We can now present a  main result of this section, its counterpart will be Theorem \ref{8caminos-gamma*}.

\begin{teo}
\label{camino*}
(i) $C$ is an 8-path in $\mathbb{Z}^2$ if and only if $\Gamma^{*}(C)$ is an arc in $\mathbb{K}^2$ with pure endpoints.

(ii) $C$ is a closed   8-curve  in $\mathbb{Z}^2$ if and only if $\Gamma^{*}(C)$ is a Jordan curve in $\mathbb{K}^2$.
\end{teo}

\begin{proof}
(i) Suppose $C = \{x_0, x_1, \ldots, x_{n-1}, x_n\}$ is an 8-path with endpoints $x_0$ and $x_n$. We must show that $\Gamma^*(C)$ is homeomorphic to an interval of the digital line $\mathbb{K}$. To do so, we will use Theorem \ref{arcosycaminos}.

By Proposition \ref{puntosfinales}, we have $|A(\Gamma(x_i)) \cap \Gamma^*(C)| = 1$ for $i = 0, n$. Furthermore, by Proposition \ref{puntosnofinales}, we have $|A(x) \cap \Gamma^*(C)| = 2$ for every $x \in \Gamma^*(C)$ such that $\Gamma(x_n) \neq x \neq \Gamma(x_0)$.

Next, let us show that $\Gamma^*(C)$ is connected. Let $z_1, \ldots, z_m$ be an enumeration of $\Gamma^*(C)$ such that $z_1 = \Gamma(x_0)$ and $z_m = \Gamma(x_n)$, where $A(z_j) \cap \Gamma^*(C) = \{z_{j-1}, z_{j+1}\}$ for $1 < j < m$. Define $A_j = (A(z_j) \cup \{z_j\}) \cap \Gamma^*(C)$ for $1 \leq j \leq m$. Then $A_j \cap A_{j+1} \neq \emptyset$ and $A_j$ is connected. Consequently, $\Gamma^*(C)$ is connected. By Theorem \ref{arcosycaminos}, we conclude that $\Gamma^*(C)$ is an arc with endpoints $\Gamma(x_0)$ and $\Gamma(x_n)$, which are also pure.

Conversely, suppose $\Gamma^{*}(C)$ is an arc with pure endpoints. Note that $\Gamma^{-1}(\Gamma^{*}(C)) = C$. Therefore, $\Gamma^{*}(\Gamma^{-1}(\Gamma^{*}(C))) = \Gamma^{*}(C)$. By Proposition \ref{suficiente_8}, it follows that $\Gamma^{-1}(\Gamma^{*}(C)) = C$ is an 8-path.

(ii) is shown as (i).
\end{proof}

\begin{propo} 
\label{8-finalarco}
Let $C$ be an arc in $\mathbb{K}^2$ with pure endpoints $z$ and $w$. If $\Gamma^{-1}(C)$ is an 8-path in $\mathbb{Z}^2$, then $\Gamma^{-1}(z)$ and $\Gamma^{-1}(w)$ are 8-endpoints in $\Gamma^{-1}(C)$.
\end{propo}

\begin{proof}
Let $u$ be an 8-endpoint of $\Gamma^{-1}(C)$. Since $\Gamma^{-1}(C)$ is an 8-path, it has exactly two 8-endpoints, so it suffices to show that $u \in \{\Gamma^{-1}(z), \Gamma^{-1}(w)\}$. Suppose $u \notin \{\Gamma^{-1}(z), \Gamma^{-1}(w)\}$ to arrive at a contradiction. As $\Gamma(u) \notin \{z, w\}$, it follows that $\Gamma(u) \cap C = \{v_1, v_2\}$ with $v_1 \neq v_2$. We consider three cases:

\medskip

\noindent \textbf{Case 1:} Assume $v_1$ and $v_2$ are pure points. By Proposition \ref{4adyacentes}, $\Gamma^{-1}(v_1)$ and $\Gamma^{-1}(v_2)$ are 8-adjacent to $u$. This contradicts the assumption that $u$ is an 8-endpoint of $\Gamma^{-1}(C)$.

\medskip

\noindent \textbf{Case 2:} Assume $v_1$ and $v_2$ are mixed points. Then neither $v_1$ nor $v_2$ are endpoints of $C$. It follows that $A(v_1) \cap C = \{\Gamma(u), y_1\}$ and $A(v_2) \cap C = \{\Gamma(u), y_2\}$, where $y_1$ and $y_2$ are pure points of $C$. 
Note that $\Gamma(u)$ and $y_1$ are not adjacent. Then, by Proposition \ref{dospuntosmixtos}, $v_1$ is the only mixed point adjacent to both. Consequently, $\Gamma^{-1}(y_1)$ is 8-adjacent to $u$ (see Proposition \ref{4adyacentes}). Similarly, $v_2$ is 8-adjacent to $u$. This contradicts the assumption that $u$ is an 8-endpoint of $\Gamma^{-1}(C)$.

\medskip

\noindent \textbf{Case 3:} Assume $v_1$ or $v_2$ is a mixed point. Using reasoning similar to the previous cases, we conclude that $u$ cannot be an 8-endpoint.

\end{proof}

The following proposition is the converse of Proposition \ref{suficiente_8}.

\begin{propo} 
\label{necesaria_8}
\begin{itemize}
\item[(i)] Let $C$ be an arc in $\mathbb{K}^2$ with pure endpoints. If $\Gamma^{-1}(C)$ is an 8-path in $\mathbb{Z}^2$, then $\Gamma^{*}(\Gamma^{-1}(C)) = C$.

\item[(ii)] If $C$ is an arc  in $\mathbb{K}^2$ with mixed endpoints and $\Gamma^{-1}(C)$ is an $8$-path, then  $\Gamma^{*}(\Gamma^{-1}(C\setminus \{z,w\}))=C\setminus\{z,w\}$ where $z$ and $w$ are the endpoints of $C$.

\item[(iii)] If $C$ is a Jordan curve in  $\mathbb{K}^2$ and $\Gamma^{-1}(C)$ is closed  8-curve in  $\mathbb{Z}^2$, then  $\Gamma^{*}(\Gamma^{-1}(C))=C$.
\end{itemize}
\end{propo}

\begin{proof}
(i) Let $z$ and $w$ be the endpoints of $C$. 
To simplify the notation, we will write $\Gamma^{-1}(C)^{*}$ instead of $\Gamma^{*}(\Gamma^{-1}(C))$. By Theorem \ref{camino*}, $\Gamma^{-1}(C)^{*}$ is an arc. First, we show that $z$ and $w$ are the endpoints of $\Gamma^{-1}(C)^{*}$. By Proposition \ref{8-finalarco}, $\Gamma^{-1}(z)$ and $\Gamma^{-1}(w)$ are 8-endpoints of $\Gamma^{-1}(C)$. By Proposition \ref{puntosfinales}, $|A(z) \cap \Gamma^{-1}(C)^{*}| = |A(w) \cap \Gamma^{-1}(C)^{*}| = 1$, which means that $z$ and $w$ are the endpoints of $\Gamma^{-1}(C)^{*}$.

Since $z$ and $w$ are pure, by Proposition \ref{1contenencia}(iii), $C \subseteq \Gamma^{-1}(C)^{*}$. 
Now we prove that $\Gamma^{-1}(C)^{*} \subseteq C$. Let $x \in \Gamma^{-1}(C)^{*}$. By Proposition \ref{1contenencia}(i), we may assume that $x$ is mixed. Suppose $x \notin C$ and arrive at a contradiction. We have already shown that $\Gamma^{-1}(C)^{*}$ is an arc with endpoints $z$ and $w$. 
Therefore, $x$ is not an endpoint of $\Gamma^{-1}(C)^{*}$, and thus there exist two pure points $w_1$ and $w_2$ such that $A(x) \cap \Gamma^{-1}(C)^{*} = \{w_1, w_2\}$. It follows that $w_1, w_2 \in C$ (see Proposition \ref{1contenencia}). We consider two cases:

\medskip

\noindent \textbf{Case 1:} Suppose $w_1, w_2 \notin \{w, z\}$. Then $A(w_1) \cap C = \{z_1, z_2\}$ for some points $z_1$ and $z_2$. Since $x \in A(w_1)$ and $x \notin C$, it follows that $x \notin \{z_1, z_2\}$. As $C \subseteq \Gamma^{-1}(C)^{*}$, we have $\{z_1, z_2, x\} \subseteq A(w_1) \cap \Gamma^{-1}(C)^{*}$, which contradicts the fact that $\Gamma^{-1}(C)^{*}$ is an arc.

\medskip

\noindent \textbf{Case 2:} Suppose $w_1 \in \{w, z\}$ or $w_2 \in \{w, z\}$. Without loss of generality, assume $w_1 = z$. 
Since $z = w_1$ is an endpoint of $C$, we have $A(w_1) \cap C = \{z_1\}$ for some $z_1$. As $x \notin C$, it follows that $x \neq z_1$. On the other hand, we have already shown that $C \subseteq \Gamma^{-1}(C)^{*}$. Thus, $A(w_1) \cap C \subseteq A(w_1) \cap \Gamma^{-1}(C)^{*}$. Consequently, $|A(w_1) \cap \Gamma^{-1}(C)^{*}| \geq 2$, since $x \neq z_1$ and $x \in A(w_1)$. This contradicts the fact that $\Gamma^{-1}(C)$ is an 8-path.

This ends the verification that $\Gamma^{-1}(C)^{*} \subseteq C$. 

\medskip 

\noindent ii) Let $C_0=C\setminus\{z,w\}$, then  $C_0$ is an arc with pure endpoints and  $\Gamma^{-1}(C_0)=\Gamma^{-1}(C)$ is an 8-path. By part i), $\Gamma^{*}(\Gamma^{-1}(C_0))=C_0$. Thus $\Gamma^{*}(\Gamma^{-1}(C\setminus \{z,w\}))=C\setminus\{z,w\}$.

\medskip 

\noindent iii)  It is shown analogously.

\end{proof}

Finally, we can present the counterpart to Theorem \ref{camino*}.

\begin{teo}
\label{8caminos-gamma*}
(i) Let $C$ be an arc in $\mathbb{K}^2$ with pure endpoints. Then $\Gamma^{-1}(C)$ is an 8-path in $\mathbb{Z}^2$ if and only if $\Gamma^{*}(\Gamma^{-1}(C)) = C$.

(ii) Let $J\subseteq  \mathbb{K}^2$ be a Jordan curve in $\mathbb{K}^2$. Then $\Gamma^{-1}(J)$ is a closed  8-curve if and only if $\Gamma^{*}(\Gamma^{-1}(J))=J$.

\end{teo}

\begin{proof}
It follows immediately from Propositions \ref{suficiente_8}, \ref{4-curva} and \ref{necesaria_8}.
\end{proof}

The following result was motivated by a remark given in \cite{Khalimsky2} (see the end of page 51) and will be needed later on. 

\begin{propo}
\label{4-caminoG}
If $C$ is a 4-path in $\mathbb{Z}^{2}$, then $\Gamma(C)$ is an arc in $\mathbb{K}^{2}$.
\end{propo}

\begin{proof}
Let $C = \{x_0, x_1, \ldots, x_{n-1}, x_n\}$ be a 4-path with 4-endpoints $x_0$ and $x_n$. To prove that $\Gamma(C)$ is an arc, we will use Theorem \ref{arcosycaminos}.

First, let us show that $\Gamma(C)$ is connected.  
Let $A_{i-1} = \{\Gamma(x_{i-1}), \Gamma(x_i)\}$ for $1 \leq i \leq n$. Clearly, $A_{i} \cap A_{i+1} \neq \emptyset$ for $0 \leq i < n$, and by Proposition \ref{4adyacentes}, each $A_{i}$ is connected. Since $\Gamma(C) = \bigcup\{A_{i-1} : 1 \leq i \leq n\}$, we conclude that $\Gamma(C)$ is connected.

Next, we show that $|A(\Gamma(x_i)) \cap \Gamma(C)| = 2$ for $1 \leq i \leq n-1$.  
Since $C$ is a 4-path and $x_i$ is not a 4-endpoint of $C$, we have that $x_i$ is 4-adjacent to both $x_{i-1}$ and $x_{i+1}$. Therefore, by Proposition \ref{4adyacentes}, $|A(\Gamma(x_i)) \cap \Gamma(C)| = 2$.

Finally, we prove that $|A(\Gamma(x_i)) \cap \Gamma(C)| = 1$ for $i = 0, n$.  
We will do it for $i = 0$, as the other case is similar. Since $x_0$ is 4-adjacent to $x_1$, $\Gamma(x_1) \in A(x_0)$ (by Proposition \ref{4adyacentes}), and necessarily $|\Gamma(x_0) \cap \Gamma(C)| = 1$, because otherwise $C$ would not be a 4-path.
\end{proof}

\section{ Preservation of connectivity}
\label{sec-prese-conex}

In this section, we analyze the preservation of connectivity.

\begin{teo}
\label{4-caminoplano}
$A \subseteq \mathbb{Z}^2$ is 4-connected if and only if $\Gamma(A)$ is connected in $\mathbb{K}^2$.
\end{teo}

\begin{proof}
($\Rightarrow$) Being 4-connected is equivalent to being connected by 4-paths (see Proposition \ref{curvasimple}). Thus, by Proposition \ref{4-caminoG} and Theorem \ref{arcosycaminos}, the result follows.

\medskip

($\Leftarrow$) Let $x, y \in A$. Since $\Gamma(A)$ is connected, there exists an arc $D \subseteq \Gamma(A)$ from $\Gamma(x)$ to $\Gamma(y)$. Note that $D$ contains only pure points; therefore, by Proposition \ref{caminor1}, $\Gamma^{-1}(D)$ is a 4-path from $x$ to $y$ in $A$. Hence, $A$ is 4-connected.
\end{proof}

\begin{propo}
\label{conexo-8}
If $C \subseteq \mathbb{K}^2$ is connected, then $\Gamma^{-1}(C)$ is 8-connected.
\end{propo}

\begin{proof}
Assume $C$ is connected but $\Gamma^{-1}(C)$ is not 8-connected. Then there exist two nonempty sets $A$ and $B$ such that $A$ and $B$ are not 8-adjacent, $\Gamma^{-1}(C) = A \cup B$, and $A \cap B = \emptyset$. Note that $\Gamma(A) \cup \Gamma(B) \subseteq C$ and $\Gamma(A) \cup \Gamma(B)$ contain all pure points of $C$.

Since $C$ is connected, for every $x \in A$ and $y \in B$, there exists an arc in $C$ from $\Gamma(x)$ to $\Gamma(y)$. Let $D \subseteq C$ be an arc of minimal cardinality between elements of $\Gamma(A)$ and $\Gamma(B)$. Let $x \in A$ and $y \in B$ such that $\Gamma(x)$ and $\Gamma(y)$ are the endpoints of $D$.

First, we show that $\Gamma(x)$ and $\Gamma(y)$ are the only pure points of $D$. Suppose otherwise, and let $w \in D$ be another pure point. Since $w \in \Gamma(A) \cup \Gamma(B)$, assume without loss of generality that $w \in \Gamma(A)$. As $D$ is an arc, there exists another arc $D' \subseteq D$ from $w$ to $\Gamma(y)$. Note that $\Gamma^{-1}(w) \in A$ and $D'$ is contained in $D \setminus \{\Gamma(x)\}$, contradicting the minimality of $D$. Hence, the only pure points of $D$ are $\Gamma(x)$ and $\Gamma(y)$.

Since $D$ is an arc with endpoints $\Gamma(x)$ and $\Gamma(y)$, we have $A(\Gamma(x)) \cap D = \{w_1\}$ and $A(\Gamma(y)) \cap D = \{w_2\}$. Observe that $w_1$ and $w_2$ are mixed points, as $\Gamma(x)$ and $\Gamma(y)$ are the only pure points of $D$. Consequently, $|A(w_1) \cap D| = |A(w_2) \cap D| = 2$. Since the adjacency of a mixed point includes only pure points, $w_1$ is also adjacent to $\Gamma(y)$, and $w_2$ is adjacent to $\Gamma(x)$. We claim that $w_1 = w_2$. Indeed, assume $w_1 \neq w_2$. By Proposition \ref{dospuntosmixtos}, $\Gamma(x)$ and $\Gamma(y)$ are adjacent. Then, by Proposition \ref{4adyacentes}, $x$ and $y$ are 4-adjacent, contradicting the assumption that $A$ and $B$ are not 8-adjacent. Thus, $w_1 = w_2$.

Since $w_1 = w_2$, $w_1$ is the only mixed point adjacent to both $\Gamma(x)$ and $\Gamma(y)$. Therefore, by Proposition \ref{4adyacentes}, $x$ and $y$ are 8-adjacent. This implies that $A$ and $B$ are 8-adjacent, which contradicts the choice of $A$ and $B$. This contradiction completes the proof.
\end{proof}

\begin{teo}
\label{8-conexida}
$A\subseteq \mathbb{Z}^{2}$ is 8-connected if and only if $\Gamma^{*}(A)$ is connected in $\mathbb{K}^{2}$.
\end{teo}

\begin{proof} 
($\Rightarrow$) Let us suppose that $A$ is 8-connected. To show that $\Gamma^*(A)$ is connected, it suffices to prove that for every pair of non-adjacent points $x, y \in \Gamma^*(A)$, there exists a path between them in $\Gamma^*(A)$. Fix $x, y \in \Gamma^*(A)$ and consider the following cases:

\medskip 

\noindent \textbf{Case 1:} Suppose that $x$ and $y$ are pure points. Then $x, y \in \Gamma(A)$, and therefore $\Gamma^{-1}(x)$, $\Gamma^{-1}(y) \in A$. Since $A$ is 8-connected, there exists an 8-path $C$ from $\Gamma^{-1}(x)$ to $\Gamma^{-1}(y)$. By Theorem \ref{camino*}, $\Gamma^*(C)$ is a path from $x$ to $y$ in $\Gamma^*(A)$.

\medskip 

\noindent \textbf{Case 2:} Suppose that $x$ is pure and $y$ is mixed. Since $y \in \Gamma^*(A)$, either $cl(y) \subseteq \Gamma^*(A)$ or $N(y) \subseteq \Gamma^*(A)$. In both cases, there exist $w_1, w_2 \in \Gamma^*(A)$ such that $w_1$ and $w_2$ are pure points adjacent to $y$. Thus, $\Gamma^{-1}(w_1), \Gamma^{-1}(w_2) \in A$. Since $A$ is 8-connected, there exists an 8-path $C \subseteq A$ from $\Gamma^{-1}(x)$ to $\Gamma^{-1}(w_1)$. By Theorem \ref{camino*}, $\Gamma^{*}(C)$ is a path from $x$ to $w_1$. Note that $\Gamma^{*}(C) \cup \{y\}$ is connected, as $y$ is adjacent to $w_1$. Therefore, there exists a path from $x$ to $y$ in $\Gamma^{*}(C) \cup \{y\} \subseteq \Gamma^*(A)$.
      
\medskip

\noindent \textbf{Case 3:} Suppose that both $x$ and $y$ are mixed points. Using reasoning analogous to the previous case, it follows that there exists a path from $x$ to $y$ in $\Gamma^*(A)$.

\medskip
 
($\Leftarrow$) Assume that $\Gamma^*(A)$ is connected. Since $A = \Gamma^{-1}(\Gamma^*(A))$, $A$ is 8-connected by Proposition \ref{conexo-8}.

\end{proof}

\begin{propo}
\label{cont_conex}
If $\Gamma(A)$ is connected, then $\Gamma^{*}(A)$ is connected.
\end{propo}

\begin{proof}
Let $x$ be a mixed point of $\Gamma^{*}(A)$. Then $N(x) \subseteq \Gamma(A) \cup \{x\}$ or $cl(x) \subseteq \Gamma(A) \cup \{x\}$. Thus, $x$ is adjacent to $\Gamma(A)$, and therefore $\Gamma(A) \cup \{x\}$ is connected. Repeating this argument for every mixed point in $\Gamma^{*}(A) \setminus \Gamma(A)$, we conclude that $\Gamma^{*}(A)$ is connected.
\end{proof}

\begin{propo}
\label{8union}
Let $A$ and $B$ be 8-connected disjoint subsets of $\mathbb{Z}^2$. Suppose there exist a mixed point $e$ and points $a, b$ such that $a \in \Gamma(A) \cap A(e)$ and $b \in \Gamma(B) \cap A(e)$. Then $A \cup B$ is 8-connected. 
\end{propo}

\begin{proof}
Since $A$ and $B$ are 8-connected, it suffices to show that $A$ and $B$ have 8-adjacent points. As $a, b \in A(e)$, we have that $b$ is adjacent to $a$ or $e$ is the only mixed point such that $a, b \in A(e)$. By Proposition \ref{4adyacentes}, it follows that $\Gamma^{-1}(a)$ and $\Gamma^{-1}(b)$ are 8-adjacent.
\end{proof}
\section{ Transformation of Components}
\label{preserva-componentes}

In this section, we analyze how to transform a component of a subset $S \subseteq \mathbb{Z}^2$ into a component of $\Gamma(S)$. It is not necessarily true that if $A$ and $B$ are the 8-components of $A \cup B$, then $\Gamma(A)$ is a component of $\Gamma(A \cup B)$. For this reason, we introduce an operator $[\cdot, \cdot]$ that transforms a pair $(A, J)$ into a subset of $\mathbb{K}^2$, where $A \subseteq \mathbb{Z}^2$ and $J \subseteq \mathbb{K}^2$:
$$
[A,J]=\{x\in \mathbb{K}^2 \mid A(x)\subseteq \Gamma(A)\cup J \text{ and } x \text{ mixed}\}.
 $$
We will write $A_J$ instead of $[A, J]$ to simplify the notation.

The set $A_J$ is used when $J$ is a Jordan curve. Let $e \in A_J$. Then, $A(e) \subseteq \Gamma(A) \cup J$. Recall that $A(e)$ is a Jordan curve; therefore, if $A(e) \subseteq J$, then $A(e) = J$. Since $e$ is mixed, $A(e)$ contains four points. Thus, if $J$ has more than four points, it follows that $A(e) \cap \Gamma(A) \neq \emptyset$, which means  that $e$ is adjacent to $\Gamma(A)$. We formalize this observation as follows:

\begin{propo}
\label{AJ} 
Let $J$ be a Jordan curve in $\mathbb{K}^{2}$ with more that four points, and let $A \subseteq \mathbb{Z}^2 \setminus \Gamma^{-1}(J)$. Every point in $A_J$ is adjacent to $\Gamma(A)$. \end{propo}

\begin{propo} 
\label{conexo1} 
Let $J$ be a Jordan curve in $\mathbb{K}^{2}$ that contains only pure points and has more than four points. If $A \subseteq \mathbb{Z}^2 \setminus \Gamma^{-1}(J)$ is $8$-connected, then $\Gamma^{*}(A) \cup A_J$ is connected and satisfies $\Gamma^{*}(A) \cup A_J \subseteq \mathbb{K}^2 \setminus J$. 
\end{propo}

\begin{proof} 
By Proposition \ref{8-conexida}, $\Gamma^{*}(A)$ is connected. Since $|J| \geq 5$, by Proposition \ref{AJ}, every point in $A_J$ is adjacent to $\Gamma(A) \subseteq \Gamma^{*}(A)$. Thus, $\Gamma^{*}(A) \cup A_J$ is connected.

It is clear that $\Gamma(A)$ consists of the pure points of $\Gamma^{*}(A) \cup A_J$, and since $J$ does not contain mixed points, it follows that $\Gamma^{*}(A) \cup A_J \subseteq \mathbb{K}^2 \setminus J$. 
\end{proof}

\begin{propo}
\label{conexo2}
Let $J$ be a Jordan curve in $\mathbb{K}^{2}$. If $A \subseteq \mathbb{Z}^2 \setminus \Gamma^{-1}(J)$ is $4$-connected, then $\Gamma^{*}(A) \cup A_J$ is connected. \end{propo}

\begin{proof}
$\Gamma^{*}(A)$ is connected by Propositions \ref{4-caminoplano} and \ref{cont_conex}. If $A_J = \emptyset$, there is nothing to prove. Otherwise, let $e \in A_J$. By Proposition \ref{AJ}, $e$ is adjacent to $\Gamma(A)$. Consequently, $\Gamma^{*}(A) \cup A_J$ is connected. 
\end{proof}

\begin{propo}
\label{disconexo2}
Let $A, B$ be subsets of $\mathbb{Z}^2$. If $A$ and $B$ are not 8-adjacent, then $\Gamma(A) \cup \Gamma(B)$ is disconnected. 
\end{propo}

\begin{proof}
Assume that $\Gamma(A) \cup \Gamma(B)$ is connected, and we will arrive at a contradiction. Let $a \in \Gamma(A)$, $b \in \Gamma(B)$, and $C \subseteq \Gamma(A) \cup \Gamma(B)$ be an arc with endpoints $a$ and $b$. Since $C$ contains only pure points, by Proposition \ref{caminor1}, $\Gamma^{-1}(C)$ is a 4-path with endpoints $\Gamma^{-1}(a) \in A$ and $\Gamma^{-1}(b) \in B$. Since $\Gamma^{-1}(C)$ is 8-connected (as it is a 4-path), it follows that $A$ and $B$ are 8-adjacent, which contradicts the hypothesis. 
\end{proof}

\begin{propo} 
\label{disconexos1}
Let $J$ be a Jordan curve in $\mathbb{K}^{2}$. If $A$ and $B$ are two subsets of $\mathbb{Z}^2 \setminus \Gamma^{-1}(J)$ that are not $8$-adjacent, then $(\Gamma^{*}(A) \cup A_J) \cup (\Gamma^{*}(B) \cup B_J)$ is disconnected. \end{propo}

\begin{proof}
We will assume that $\Gamma^{*}(A) \cup A_J \cup \Gamma^{*}(B) \cup B_J$ is connected and derive a contradiction. Let $a \in \Gamma^{*}(A) \cup A_J$, $b \in \Gamma^{*}(B) \cup B_J$, and $C \subseteq \Gamma^{*}(A) \cup A_J \cup \Gamma^{*}(B) \cup B_J$ be an arc with endpoints $a$ and $b$. Consider the following cases:

\medskip 

\noindent \textbf{Case 1:} Suppose $a$ and $b$ are pure points. Then $a \in \Gamma(A)$ and $b \in \Gamma(B)$. By Proposition \ref{conexo-8}, $\Gamma^{-1}(C)$ is $8$-connected, and thus $A$ and $B$ are $8$-adjacent, which is a contradiction.

\medskip

\noindent\textbf{Case 2:} Suppose $a$ and $b$ are mixed points. Then $C \setminus \{a, b\}$ is an arc with pure endpoints $a_1$ and $b_1$, where $a_1$ is adjacent to $a$ and $b_1$ to $b$. It suffices to show that $a_1 \in \Gamma(A)$ and $b_1 \in \Gamma(B)$, since this would imply (as in Case 1) that $A$ and $B$ are $8$-adjacent, leading to a contradiction.

Suppose instead that $a_1 \notin \Gamma(A)$ and $b_1 \notin \Gamma(B)$. Since $C \subseteq \Gamma^{*}(A) \cup A_J \cup \Gamma^{*}(B) \cup B_J$ and $a_1$ and $b_1$ are pure points, we have $a_1 \in \Gamma(B)$ and $b_1 \in \Gamma(A)$.

We claim that $A(a) \cap \Gamma(A) = \emptyset$. Indeed, if $A(a) \cap \Gamma(A) \neq \emptyset$, then $A \cup B$ would be $8$-connected (by Proposition \ref{8union}, since $a_1 \in A(a) \cap \Gamma(B)$); however, this is impossible, as $A$ and $B$ are $8$-components. Therefore, $A(a) \cap \Gamma(A) = \emptyset$.

Moreover, $A(a)$ is a Jordan curve with 4 points, while $J$ is a Jordan curve with more than five points. Hence, $A(a) \not\subseteq J$, and thus $a \notin A_J$. Since $a \in \Gamma^{*}(A) \cup A_J$, it follows that $N(a) \subseteq \Gamma(A) \cup \{a\}$ or $cl(a) \subseteq \Gamma(A) \cup \{a\}$. As $a_1 \in A(a) = N(a) \cup cl(a)$, we conclude that $a_1 \in \Gamma(A)$, which contradicts our assumption.

Thus, we have shown that $a_1 \in \Gamma(A)$. Similarly, we obtain that $b_1 \in \Gamma(B)$.

\medskip

\noindent\textbf{Case 3:} Suppose $a$ is pure and $b$ is mixed. Then $C \setminus \{b\}$ is an arc with pure endpoints. Using reasoning analogous to Case 2, we conclude that $A$ and $B$ are $8$-adjacent, which is a contradiction.

\end{proof}

\begin{propo}
\label{disconexos3}
Let $J$ be a Jordan curve in $\mathbb{K}^{2}$ such that  $\Gamma^{*}(\Gamma^{-1}(J)) = J$. If $A$ and $B$ are two 4-components of $\mathbb{Z}^2 \setminus \Gamma^{-1}(J)$ such that $\Gamma^{*}(A) \cup A_J \cup \Gamma^{*}(B) \cup B_J \subseteq \mathbb{K}^2 \setminus J$, then $\Gamma^{*}(A) \cup A_J \cup \Gamma^{*}(B) \cup B_J$ is disconnected.
\end{propo}

\begin{proof}
We will assume that $\Gamma^{*}(A) \cup A_J \cup \Gamma^{*}(B) \cup B_J$ is connected and reach a contradiction. Let $a_1 \in \Gamma(A)$, $b_1 \in \Gamma(B)$, and $C \subseteq \Gamma^{*}(A) \cup A_J \cup \Gamma^{*}(B) \cup B_J$ be an arc with endpoints $a_1$ and $b_1$. We will show that we can obtain an arc $C$ consisting solely of pure points. This is impossible, as it would imply that $C \subseteq \Gamma(A) \cup \Gamma(B)$, and by Proposition \ref{caminor1}, $\Gamma^{-1}(C)$ would be a 4-path connecting $A$ and $B$. This contradicts the fact that $A$ and $B$ are 4-components of $\mathbb{Z}^2 \setminus \Gamma^{-1}(J)$.

The strategy is to start with an arbitrary arc $C$ and replace each mixed point in $C$ with a pure point. Let $m \in C$ be a mixed point. We will construct an arc $C_1$ of the form $C \setminus \{m\} \cup \{w\}$ from $a_1$ to $b_1$ such that $C_1 \subseteq \Gamma^{*}(A) \cup A_J \cup \Gamma^{*}(B) \cup B_J$. Using a similar argument, we replace each mixed point in $C$ with a pure point and construct an arc $C'$ from $a_1$ to $b_1$ such that $C'$ consists only of pure points and $C' \subseteq \Gamma(A) \cup \Gamma(B)$. As noted above, this leads to a contradiction.

Let $m \in C$ be a mixed point. Let $w_1, w_2, w_3, w_4$ be such that $A(m) \cap C = \{w_1, w_2\}$ and $A(m) \setminus C = \{w_3, w_4\}$. Since $C$ is an arc and does not loop around a mixed point, it follows that
$$
N(m) = \{w_3, w_4, m\} \quad \text{or} \quad cl(m) = \{w_3, w_4, m\}.
$$
We claim that either $w_3$ or $w_4$ is in $\Gamma(A) \cup \Gamma(B)$. Indeed, if this were not the case,  $w_3$ and $w_4$ must belong to $J$ (since $\{w_3, w_4\} \cap (\Gamma(A) \cup \Gamma(B)) = \emptyset$). Thus, $N(m) \subseteq \Gamma(\Gamma^{-1}(J)) \cup \{m\}$ or $cl(m) \subseteq \Gamma(\Gamma^{-1}(J)) \cup \{m\}$. Hence, $m \in \Gamma^{*}(\Gamma^{-1}(J)) = J$, which is a contradiction, as $m \in C$ and  $C \subseteq  \mathbb{K}^2 \setminus J$. This shows that either $w_3$ or $w_4$ must be in $\Gamma(A) \cup \Gamma(B)$.

Assume that $w_3 \in \Gamma(A) \cup \Gamma(B)$; the other case is analogous. Let $C_1 = (C \setminus \{m\}) \cup \{w_3\}$. Then $C_1$ is an arc from $a_1$ to $b_1$ such that $C_1 \subseteq \Gamma^{*}(A) \cup A_J \cup \Gamma^{*}(B) \cup B_J$.
\end{proof}

\section{ Jordan curve theorem for $\mathbb{K}^2$}
\label{sec-curvaJordan}

Recall that the Jordan theorem in the plane $\mathbb{K}^2$ states that if $J$ is a Jordan curve in $\mathbb{K}^2$, then $\mathbb{K}^2 \setminus J$ has two components. The approach we will use to prove the theorem is based the analogous theorem for $\mathbb{Z}^2$ (Theorem \ref{Jordank}) applied to $\Gamma^{-1}(J)$. Suppose for a moment that $\Gamma^{-1}(J)$ is a  closed 4-curve and let $A$, $B$ be the 8-components of $\mathbb{Z}^2 \setminus \Gamma^{-1}(J)$. It is not true that $\Gamma(A)$ and $\Gamma(B)$ are the components of $\mathbb{K} \setminus J$; they may not even be connected, as illustrated in Figure \ref{contraejemplo1}. Resolving this technical issue required developing the ideas presented in the previous sections. We will show that, under some conditions, the following holds:
\begin{equation} \label{componentes} \mathbb{K}^{2} \setminus J = (\Gamma^{*}(A) \cup A_J) \, \cup \, (\Gamma^{*}(B) \cup B_J) \end{equation}
and moreover, $\Gamma^{*}(A) \cup A_J$ and $\Gamma^{*}(B) \cup B_J$ are in fact the two components of $\mathbb{K}^{2} \setminus J$.

\bigskip

\begin{figure}[h]
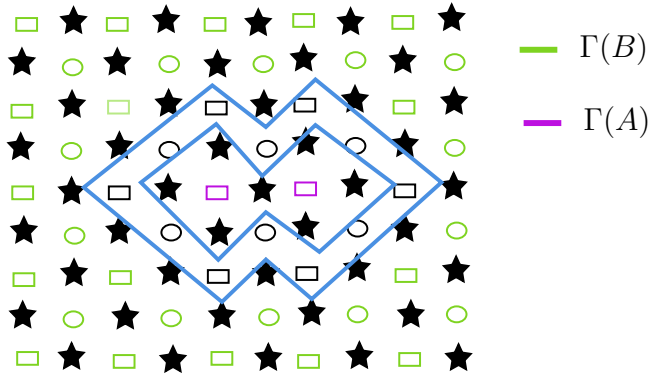

    \centering
\tikzset{every picture/.style={line width=0.75pt}} 


    \caption{$A$ and $B$ are the 8-components of $\mathbb{Z}^2\setminus \Gamma^{-1}(J)$, but  $\Gamma (A)$ is not connected.}
    \label{contraejemplo1}
\end{figure}

We begin by analyzing Jordan curves composed solely of pure points

\begin{teo} \label{solopuros} If $J$ is a Jordan curve in $\mathbb{K}^2$ consisting only of pure points and $|J| >4$, then $\mathbb{K}^2 \setminus J$ has two components. Moreover, each  point of $J$ is adjacent to both components of $\mathbb{K}^2 \setminus J$.
\end{teo}

\begin{proof} By Proposition \ref{4-curva}, $\Gamma^{-1}(J)$ is a  closed  4-curve. By Theorem \ref{Jordank}, let $A$ and $B$ the 8-components of $\mathbb{Z}^2 \setminus \Gamma^{-1}(J)$. We will show that $\Gamma^{*}(A) \cup A_J$ and $\Gamma^{*}(B) \cup B_J$ are the components of $\mathbb{K}^2 \setminus J$. 

By Proposition \ref{conexo1}, $\Gamma^{*}(A) \cup A_J$ and $\Gamma^{*}(B) \cup B_J$ are connected. We now prove \eqref{componentes}. Note that $\mathbb{K}^2 \setminus J$ and $\Gamma^{*}(A) \cup A_J \cup \Gamma^{*}(B) \cup B_J$ share the same pure points. Therefore, to verify equality, we only need to analyze the mixed points.

The inclusion $\supseteq$ in \eqref{componentes} was demonstrated in Proposition \ref{conexo1}. To verify $\subseteq$, let $e \in \mathbb{K}^2 \setminus J$  be a mixed point. Consider the following cases:

\medskip

\noindent\textbf{Case 1:} Suppose that one of the following conditions holds:

\begin{enumerate} \item $N(e) \subseteq \Gamma(A) \cup \{e\}$

\item $cl(e) \subseteq \Gamma(A) \cup \{e\}$

\item $cl(e) \subseteq \Gamma(B) \cup \{e\}$

\item $N(e) \subseteq \Gamma(B) \cup \{e\}$ \end{enumerate}

All four alternatives are analyzed in the same way. Without loss of generality, suppose $N(e) \subseteq \Gamma(A) \cup \{e
\}$. Then $e \in \Gamma^{*}(A)$, and therefore $e \in \Gamma^{*}(A) \cup A_J \cup \Gamma^{*}(B) \cup B_J$.

\medskip

\noindent\textbf{Case 2:} Suppose that none of the four conditions mentioned in Case 1 holds. We will show that $e \in A_J$ or $e \in B_J$. To simplify the proof, we first establish two facts.

\medskip

\begin{itemize}
\item [(i)] $A(e) \cap \Gamma(A) \neq \emptyset$ or $A(e) \cap \Gamma(B) \neq \emptyset$. Indeed, $A(e)$ contains only pure points, so $
A(e)\subseteq \Gamma(\mathbb{Z}^2)= \Gamma(A)\cup \Gamma(B)\cup J$.
Since $|J| \geq 5$, we have $A(e) \nsubseteq J$. Thus, $A(e) \cap (\Gamma(A) \cup \Gamma(B)) \neq \emptyset$.

\medskip

\item [(ii)] If $A(e) \cap \Gamma(A) \neq \emptyset$, then $\Gamma(B) \cap A(e) = \emptyset$. Suppose $A(e) \cap \Gamma(A) \neq \emptyset$ and $\Gamma(B) \cap A(e) \neq \emptyset$; we will derive a contradiction. Let $A(e) = \{a, b, c, d\}$, where
$$
N(e)=\{a,b,e\} \;\text{and}\; cl(e)=\{c,d,e\}.
$$
Assume $a \in \Gamma(A)$. Then $c, d \notin \Gamma(B)$, because $\Gamma(A) \cup \Gamma(B)$ is disconnected (by Proposition \ref{disconexo2}) and $cl(e)$ is adjacent to $N(e)$. Therefore, $b \in \Gamma(B)$. Since $e$ is the only mixed point adjacent to $a$ and $b$, by Proposition \ref{4adyacentes}, $\Gamma^{-1}(a)$ and $\Gamma^{-1}(b)$ are $8$-adjacent. This implies that $A$ and $B$ are $8$-adjacent, which contradicts the fact that they are the $8$-components of $\mathbb{Z}^2 \setminus \Gamma^{-1}(J)$.


\end{itemize}

From (i) and (ii), we have $A(e) \subseteq \Gamma(A) \cup J$ or $A(e) \subseteq \Gamma(B) \cup J$. By the definition of $A_J$ and $B_J$, it follows that $e \in A_J$ or $e \in B_J$. We have shown that \eqref{componentes} is valid.

To conclude that $\mathbb{K}^2 \setminus J$ has two components, it only remains to show that $\Gamma^{*}(A) \cup A_J$ and $\Gamma^{*}(B) \cup B_J$ are the components of $\mathbb{K}^2 \setminus J$. This follows from Proposition \ref{disconexos1} and the fact that $\mathbb{K}^2 \setminus J = \Gamma^{*}(A) \cup A_J \cup \Gamma^{*}(B) \cup B_J$ and that $\Gamma^{*}(A) \cup A_J$ and $\Gamma^{*}(B) \cup B_J$ are connected.

Finally, we will verify the second statement. Let $x \in J$. By Theorem \ref{Jordank}, there exist two points $a, b$ in $A$ and $B$, respectively, that are 8-adjacent to $\Gamma^{-1}(x)$. Without loss of generality, we will assume that $a \in A$. We will consider the following cases.

\textbf{Case 1}: $a$ is 4-adjacent to $\Gamma^{-1}(x)$. By Proposition \ref{4adyacentes}, $\Gamma(a) \in A(x)$, and we have already seen that $\Gamma(a) \in \Gamma^{*}(A) \cup A_J$, which is a component of $\mathbb{K}^2 \setminus J$.

\textbf{Case 2}: $a$ is 8-adjacent to $\Gamma^{-1}(x)$ but not 4-adjacent. By Proposition \ref{4adyacentes}, there exists a mixed point $d'$ such that $\Gamma(a), x \in A(d')$. We will prove that $d' \in A_J$. To do so, we need to show that $A(d') \subseteq J \cup \Gamma(A)$. Let $d, c$ be such that $A(d') = \{x, d, \Gamma(a), c\}$. By part (1) of Proposition \ref{4adyacentes}, $x$ and $\Gamma(a)$ are not adjacent; therefore, $d$ and $c$ are adjacent to both $x$ and $\Gamma(a)$. Since $c, d$ are pure points, $c = \Gamma(s)$ and $d = \Gamma(t)$ for some $s, t \in \mathbb{Z}^2$. As $a \in A$, it follows that $s, t \in A \cup \Gamma^{-1}(J)$, since $A$ and $B$ are not 8-adjacent. Thus, $d, c \in \Gamma(A) \cup J$. Therefore, $d' \in A_J$.

\end{proof}

\begin{figure}[h]
    \centering
 
\tikzset{every picture/.style={line width=0.75pt}} 

\begin{tikzpicture}[x=0.75pt,y=0.75pt,yscale=-1,xscale=1]

\draw  [fill={rgb, 255:red, 227; green, 240; blue, 247 }  ,fill opacity=0.17 ][dash pattern={on 0.84pt off 2.51pt}] (454.21,44.89) -- (539.26,125.89) -- (468.26,198.89) -- (381.26,126.73) -- cycle ;
\draw  [color={rgb, 255:red, 0; green, 0; blue, 0 }  ,draw opacity=1 ][fill={rgb, 255:red, 74; green, 144; blue, 226 }  ,fill opacity=1 ] (450.84,116.49) -- (463.37,116.49) -- (463.37,125.38) -- (450.84,125.38) -- cycle ;
\draw  [fill={rgb, 255:red, 74; green, 144; blue, 226 }  ,fill opacity=1 ] (512.97,117.56) -- (525.5,117.56) -- (525.5,126.46) -- (512.97,126.46) -- cycle ;
\draw  [fill={rgb, 255:red, 255; green, 255; blue, 255 }  ,fill opacity=1 ] (483.3,151.04) .. controls (483.3,147.98) and (485.94,145.49) .. (489.19,145.49) .. controls (492.45,145.49) and (495.09,147.98) .. (495.09,151.04) .. controls (495.09,154.11) and (492.45,156.6) .. (489.19,156.6) .. controls (485.94,156.6) and (483.3,154.11) .. (483.3,151.04) -- cycle ;
\draw  [fill={rgb, 255:red, 180; green, 191; blue, 203 }  ,fill opacity=1 ] (454.72,176.37) -- (467.25,176.37) -- (467.25,185.26) -- (454.72,185.26) -- cycle ;
\draw  [fill={rgb, 255:red, 0; green, 0; blue, 0 }  ,fill opacity=1 ] (460.06,141.23) -- (462.34,146.94) -- (467.42,147.85) -- (463.74,152.3) -- (464.61,158.58) -- (460.06,155.61) -- (455.51,158.58) -- (456.38,152.3) -- (452.7,147.85) -- (457.79,146.94) -- cycle ;
\draw  [fill={rgb, 255:red, 0; green, 0; blue, 0 }  ,fill opacity=1 ] (489.72,108.41) -- (491.99,114.13) -- (497.08,115.04) -- (493.4,119.49) -- (494.26,125.76) -- (489.72,122.8) -- (485.17,125.76) -- (486.04,119.49) -- (482.36,115.04) -- (487.44,114.13) -- cycle ;
\draw  [color={rgb, 255:red, 0; green, 0; blue, 0 }  ,draw opacity=1 ][fill={rgb, 255:red, 192; green, 206; blue, 222 }  ,fill opacity=1 ] (393.65,119.27) -- (406.18,119.27) -- (406.18,128.16) -- (393.65,128.16) -- cycle ;
\draw  [fill={rgb, 255:red, 199; green, 209; blue, 222 }  ,fill opacity=1 ] (422.96,152.33) .. controls (422.96,149.26) and (425.6,146.77) .. (428.85,146.77) .. controls (432.11,146.77) and (434.75,149.26) .. (434.75,152.33) .. controls (434.75,155.39) and (432.11,157.88) .. (428.85,157.88) .. controls (425.6,157.88) and (422.96,155.39) .. (422.96,152.33) -- cycle ;
\draw  [fill={rgb, 255:red, 0; green, 0; blue, 0 }  ,fill opacity=1 ] (425.53,113.23) -- (427.8,118.94) -- (432.89,119.86) -- (429.21,124.3) -- (430.08,130.58) -- (425.53,127.62) -- (420.98,130.58) -- (421.85,124.3) -- (418.17,119.86) -- (423.25,118.94) -- cycle ;
\draw  [fill={rgb, 255:red, 195; green, 207; blue, 220 }  ,fill opacity=1 ] (450,59.14) -- (462.53,59.14) -- (462.53,68.03) -- (450,68.03) -- cycle ;
\draw  [color={rgb, 255:red, 0; green, 0; blue, 0 }  ,draw opacity=1 ][fill={rgb, 255:red, 255; green, 255; blue, 255 }  ,fill opacity=1 ][line width=0.75]  (483.35,93.99) .. controls (483.35,90.92) and (485.99,88.43) .. (489.25,88.43) .. controls (492.5,88.43) and (495.14,90.92) .. (495.14,93.99) .. controls (495.14,97.05) and (492.5,99.54) .. (489.25,99.54) .. controls (485.99,99.54) and (483.35,97.05) .. (483.35,93.99) -- cycle ;
\draw  [fill={rgb, 255:red, 0; green, 0; blue, 0 }  ,fill opacity=1 ] (458.68,87.18) -- (460.95,92.89) -- (466.04,93.81) -- (462.36,98.25) -- (463.23,104.53) -- (458.68,101.57) -- (454.13,104.53) -- (455,98.25) -- (451.32,93.81) -- (456.4,92.89) -- cycle ;
\draw  [fill={rgb, 255:red, 188; green, 201; blue, 217 }  ,fill opacity=1 ] (418.23,94.06) .. controls (418.23,90.99) and (420.87,88.5) .. (424.13,88.5) .. controls (427.38,88.5) and (430.02,90.99) .. (430.02,94.06) .. controls (430.02,97.13) and (427.38,99.62) .. (424.13,99.62) .. controls (420.87,99.62) and (418.23,97.13) .. (418.23,94.06) -- cycle ;
\draw  [draw opacity=0][fill={rgb, 255:red, 255; green, 255; blue, 255 }  ,fill opacity=1 ] (75.06,52.41) -- (201.26,52.41) -- (201.26,172.64) -- (75.06,172.64) -- cycle ; \draw   (75.06,52.41) -- (75.06,172.64)(134.27,52.41) -- (134.27,172.64)(193.48,52.41) -- (193.48,172.64) ; \draw   (75.06,52.41) -- (201.26,52.41)(75.06,111.62) -- (201.26,111.62)(75.06,170.84) -- (201.26,170.84) ; \draw    ;
\draw  [color={rgb, 255:red, 0; green, 0; blue, 0 }  ,draw opacity=1 ][fill={rgb, 255:red, 207; green, 215; blue, 227 }  ,fill opacity=1 ][line width=0.75]  (182.35,170.84) .. controls (182.35,165.02) and (187.34,160.31) .. (193.48,160.31) .. controls (199.63,160.31) and (204.61,165.02) .. (204.61,170.84) .. controls (204.61,176.65) and (199.63,181.36) .. (193.48,181.36) .. controls (187.34,181.36) and (182.35,176.65) .. (182.35,170.84) -- cycle ;
\draw  [color={rgb, 255:red, 0; green, 0; blue, 0 }  ,draw opacity=1 ][fill={rgb, 255:red, 255; green, 255; blue, 255 }  ,fill opacity=1 ] (182.35,111.62) .. controls (182.35,105.81) and (187.34,101.1) .. (193.48,101.1) .. controls (199.63,101.1) and (204.61,105.81) .. (204.61,111.62) .. controls (204.61,117.44) and (199.63,122.15) .. (193.48,122.15) .. controls (187.34,122.15) and (182.35,117.44) .. (182.35,111.62) -- cycle ;
\draw  [fill={rgb, 255:red, 214; green, 222; blue, 232 }  ,fill opacity=1 ] (64.79,52.41) .. controls (64.79,46.6) and (69.39,41.89) .. (75.06,41.89) .. controls (80.73,41.89) and (85.33,46.6) .. (85.33,52.41) .. controls (85.33,58.22) and (80.73,62.93) .. (75.06,62.93) .. controls (69.39,62.93) and (64.79,58.22) .. (64.79,52.41) -- cycle ;
\draw  [fill={rgb, 255:red, 212; green, 220; blue, 230 }  ,fill opacity=1 ] (64.79,111.62) .. controls (64.79,105.81) and (69.39,101.1) .. (75.06,101.1) .. controls (80.73,101.1) and (85.33,105.81) .. (85.33,111.62) .. controls (85.33,117.43) and (80.73,122.15) .. (75.06,122.15) .. controls (69.39,122.15) and (64.79,117.43) .. (64.79,111.62) -- cycle ;
\draw  [fill={rgb, 255:red, 74; green, 144; blue, 226 }  ,fill opacity=1 ] (124,111.62) .. controls (124,105.81) and (128.6,101.1) .. (134.27,101.1) .. controls (139.94,101.1) and (144.54,105.81) .. (144.54,111.62) .. controls (144.54,117.43) and (139.94,122.15) .. (134.27,122.15) .. controls (128.6,122.15) and (124,117.43) .. (124,111.62) -- cycle ;
\draw  [fill={rgb, 255:red, 74; green, 144; blue, 226 }  ,fill opacity=1 ] (179.33,52.58) .. controls (179.33,46.77) and (183.93,42.06) .. (189.6,42.06) .. controls (195.27,42.06) and (199.87,46.77) .. (199.87,52.58) .. controls (199.87,58.39) and (195.27,63.1) .. (189.6,63.1) .. controls (183.93,63.1) and (179.33,58.39) .. (179.33,52.58) -- cycle ;
\draw  [fill={rgb, 255:red, 255; green, 255; blue, 255 }  ,fill opacity=1 ] (122.47,52.58) .. controls (122.47,46.77) and (127.07,42.06) .. (132.74,42.06) .. controls (138.41,42.06) and (143.01,46.77) .. (143.01,52.58) .. controls (143.01,58.39) and (138.41,63.1) .. (132.74,63.1) .. controls (127.07,63.1) and (122.47,58.39) .. (122.47,52.58) -- cycle ;
\draw  [color={rgb, 255:red, 0; green, 0; blue, 0 }  ,draw opacity=1 ][fill={rgb, 255:red, 235; green, 237; blue, 239 }  ,fill opacity=1 ] (123.14,170.84) .. controls (123.14,165.02) and (128.12,160.31) .. (134.27,160.31) .. controls (140.42,160.31) and (145.4,165.02) .. (145.4,170.84) .. controls (145.4,176.65) and (140.42,181.36) .. (134.27,181.36) .. controls (128.12,181.36) and (123.14,176.65) .. (123.14,170.84) -- cycle ;
\draw  [color={rgb, 255:red, 0; green, 0; blue, 0 }  ,draw opacity=1 ][fill={rgb, 255:red, 211; green, 217; blue, 223 }  ,fill opacity=1 ] (62.74,165.9) .. controls (62.74,160.08) and (67.73,155.37) .. (73.87,155.37) .. controls (80.02,155.37) and (85,160.08) .. (85,165.9) .. controls (85,171.71) and (80.02,176.42) .. (73.87,176.42) .. controls (67.73,176.42) and (62.74,171.71) .. (62.74,165.9) -- cycle ;
\draw    (263.26,102.89) .. controls (309.55,80.24) and (309.28,86.69) .. (346.53,100.27) ;
\draw [shift={(348.26,100.89)}, rotate = 199.75] [color={rgb, 255:red, 0; green, 0; blue, 0 }  ][line width=0.75]    (10.93,-3.29) .. controls (6.95,-1.4) and (3.31,-0.3) .. (0,0) .. controls (3.31,0.3) and (6.95,1.4) .. (10.93,3.29)   ;

\draw (482,132.4) node [anchor=north west][inner sep=0.75pt]  [font=\tiny]  {$e$};
\draw (457.84,129.78) node [anchor=north west][inner sep=0.75pt]  [font=\scriptsize]  {$d$};
\draw (514.97,129.86) node [anchor=north west][inner sep=0.75pt]  [font=\scriptsize]  {$c$};
\draw (497.14,97.39) node [anchor=north west][inner sep=0.75pt]  [font=\scriptsize]  {$a$};
\draw (497.09,154.44) node [anchor=north west][inner sep=0.75pt]  [font=\scriptsize]  {$b$};
\draw (146.54,115.02) node [anchor=north west][inner sep=0.75pt]  [font=\footnotesize]  {$p$};
\draw (85.54,63.02) node [anchor=north west][inner sep=0.75pt]  [font=\footnotesize]  {$p_{1}$};
\draw (140.54,65.02) node [anchor=north west][inner sep=0.75pt]  [font=\footnotesize]  {$p_{2}$};
\draw (198.54,67.02) node [anchor=north west][inner sep=0.75pt]  [font=\footnotesize]  {$p_{3}$};
\draw (200.54,119.02) node [anchor=north west][inner sep=0.75pt]  [font=\footnotesize]  {$p_{4}$};
\draw (205.54,167.02) node [anchor=north west][inner sep=0.75pt]  [font=\footnotesize]  {$p_{5}$};
\draw (152.54,168.02) node [anchor=north west][inner sep=0.75pt]  [font=\footnotesize]  {$p_{6}$};
\draw (87,169.3) node [anchor=north west][inner sep=0.75pt]  [font=\footnotesize]  {$p_{7}$};
\draw (87.33,115.02) node [anchor=north west][inner sep=0.75pt]  [font=\footnotesize]  {$p_{8}$};
\draw (300,57.4) node [anchor=north west][inner sep=0.75pt]    {$\Gamma $};

\end{tikzpicture}
    \caption{Case 1.2 in the proof of  Theorem \ref{mixtoconcodiccion}}
    \label{8-vencidadcasoc,d}
\end{figure}
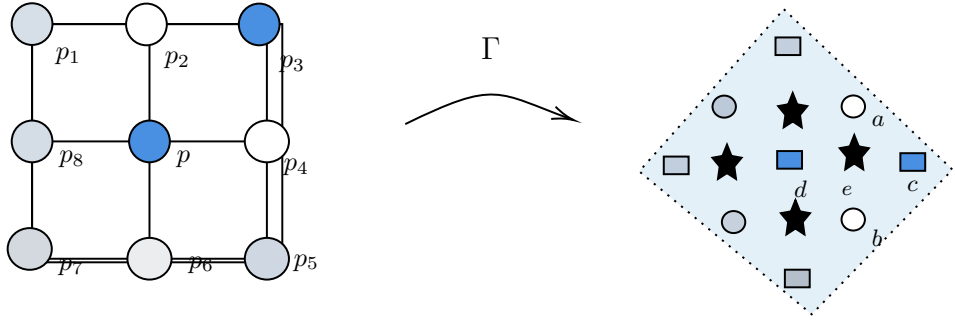

Now we will examine Jordan curves that contain mixed points. Recall that in Proposition \ref{1contenencia}, we established that
$J \subseteq \Gamma^{*}(\Gamma^{-1}(J))$ for any Jordan curve $J$. Our subsequent analysis focuses on the set $\Gamma^{*}(\Gamma^{-1}(J)) \setminus J$. We begin by considering the case where $\Gamma^{*}(\Gamma^{-1}(J)) \setminus J$ is empty.

\begin{teo}
\label{mixtoconcodiccion}
Let $J$ be a Jordan curve with more than four points. If $\Gamma^{*}(\Gamma^{-1}(J)) = J$, then $\mathbb{K}^2 \setminus J$ has two components. Moreover, each  point of $J$ is adjacent to both components of $\mathbb{K}^2 \setminus J$. 
\end{teo}

\begin{proof}
By Proposition \ref{8-curva}, we know that $\Gamma^{-1}(J)$ is a  closed  8-curve. By Theorem \ref{Jordank}, $\mathbb{Z}^2 \setminus \Gamma^{-1}(J)$ has two 4-components, denoted as $A$ and $B$. By Proposition \ref{4-caminoplano}, $\Gamma(A)$ and $\Gamma(B)$ are connected, and by Proposition \ref{cont_conex}, we have that $\Gamma^{*}(A)$ and $\Gamma^{*}(B)$ are also connected. We will show that $\Gamma^{*}(A) \cup A_J$ and $\Gamma^{*}(B) \cup B_J$ are the components of $\mathbb{K}^2 \setminus J$.

By Proposition \ref{conexo2}, we know that $\Gamma^{*}(A) \cup A_J$ and $\Gamma^{*}(B) \cup B_J$ are connected. We will now show that \eqref{componentes} holds. As in the proof of Theorem \ref{solopuros}, it suffices to prove equality for the mixed points.

\medskip

\noindent Let us prove the inclusion $\subseteq$ in \eqref{componentes}. Let $e \in \mathbb{K}^2 \setminus J$ be a mixed point. 

\medskip

Let $a, b, c,$ and $d$ be such that
\[
N(e) = \{e, a, b\} \quad \text{and} \quad cl(e) = \{e, c, d\}.
\]
Recall that since $e$ is mixed, $a, b, c,$ and $d$ are pure points. We consider two cases:

\medskip

\noindent\textbf{Case 1:} Assume $N(e) \subseteq \mathbb{K}^2 \setminus J$ or $cl(e) \subseteq \mathbb{K}^2 \setminus J$. Without loss of generality, suppose $N(e) \subseteq \mathbb{K}^2 \setminus J$. Since $\mathbb{Z}^2 \setminus \Gamma^{-1}(J) = A \cup B$ and $a, b$ are pure points, it follows that $\{a, b\} \subseteq \Gamma(A) \cup \Gamma(B)$. We consider the following subcases:

\medskip 

\textbf{1.1:} Suppose $\{a, b\} \subseteq \Gamma(A)$ or $\{a, b\} \subseteq \Gamma(B)$. In the first case, $N(e) \subseteq \Gamma(A) \cup \{e\}$, so $e \in \Gamma^{*}(A)$. Similarly, in the second case, $e \in \Gamma^{*}(B)$.

\medskip

\textbf{1.2:} Suppose $a \in \Gamma(A)$ and $b \in \Gamma(B)$. We will show that this subcase is impossible. First, observe that $c \in J$. Notice that $c \in A(a) \cap A(b)$. Then, by Proposition \ref{4adyacentes}, $\Gamma^{-1}(c)$ is 4-adjacent to both $\Gamma^{-1}(a)$ and $\Gamma^{-1}(b)$. Since $A$ and $B$ are 4-components of $\mathbb{Z}^2 \setminus \Gamma^{-1}(J)$, we must have $\Gamma^{-1}(c) \in \Gamma^{-1}(J)$, i.e., $c \in J$. Similarly, we show that $d \in J$.

Since $c, d \in J$, we have $cl(e) \subseteq \Gamma(\Gamma^{-1}(J)) \cup \{e\}$. Since $J = \Gamma^{*}(\Gamma^{-1}(J))$, it follows that $e \in J$, contradicting the assumption about $e$. Thus, this subcase is impossible.

\medskip

\noindent\textbf{Case 2:} Suppose $N(e) \nsubseteq \mathbb{K}^2 \setminus J$ and $cl(e) \nsubseteq \mathbb{K}^2 \setminus J$. Then $N(e) \cap J \neq \emptyset$ and $cl(e) \cap J \neq \emptyset$. Observe that if $|N(e) \cap J| \geq 2$, then $e \in \Gamma^{*}(\Gamma^{-1}(J))$, which is not possible since $J = \Gamma^{*}(\Gamma^{-1}(J))$ and $e \notin J$. Therefore, $|N(e) \cap J| = 1$. Similarly, we show that $|cl(e) \cap J| = 1$.
Then, $N(e) \cap J = \{b\}$ and $cl(e) \cap J = \{d\}$. Let us prove that $e \in A_J$ or $e \in B_J$. Note that $a$ and $c$ are adjacent, i.e., $\{c, a\}$ is connected, and $\{a, c\} \subseteq \mathbb{K}^{2} \setminus J$. Since $a$ and $c$ are pure points, it follows that $\{c, a\} \subseteq \Gamma(A) \cup \Gamma(B)$. Given that $\Gamma(A \cup B) = \Gamma(A) \cup \Gamma(B)$ is disconnected (see Proposition \ref{4-caminoplano}), we must have $\{c, a\} \subseteq \Gamma(A)$ or $\{c, a\} \subseteq \Gamma(B)$. Thus, $A(e) = \{a, b, c, d\} \subseteq \Gamma(A) \cup J$ or $A(e) \subseteq \Gamma(B) \cup J$. This shows that $e \in A_J$ or $e \in B_J$.

\medskip

We have demonstrated that $\mathbb{K}^2 \setminus J \subseteq \Gamma^{*}(A) \cup A_J \cup \Gamma^*(B) \cup B_J$.

\bigskip

\noindent Let us now prove that $\supseteq$ holds in \eqref{componentes}. Let $e \in \Gamma^{*}(A) \cup A_J \cup \Gamma^*(B) \cup B_J$ be a mixed point. Let $a, b, c,$ and $d$ be such that 
\[
N(e) = \{e, a, b\} \quad \text{and} \quad cl(e) = \{e, c, d\}.
\]
We will consider two cases.

\medskip 

\noindent\textbf{Case 1:} Suppose $e \in \Gamma^{*}(A) \cup \Gamma^{*}(B)$. Without loss of generality, assume $e \in \Gamma^{*}(A)$. By the definition of $\Gamma^{*}$, we have $N(e) \subseteq \Gamma(A) \cup \{e\}$ or $cl(e) \subseteq \Gamma(A) \cup \{e\}$. Again, without loss of generality, assume $N(e) \subseteq \Gamma(A) \cup \{e\}$. In summary, we can assume that $a, b \in \Gamma(A)$. Since $c, d$ are pure points, consider the following subcases:

\medskip 

\textbf{1.1:} Suppose $c, d \in \Gamma(A)$. Note that $a, b, c, d \notin \Gamma(\Gamma^{-1}(J))$. Hence, $N(e) \nsubseteq \Gamma(\Gamma^{-1}(J)) \cup \{e\}$ and $cl(e) \nsubseteq \Gamma(\Gamma^{-1}(J)) \cup \{e\}$. It follows that $e \notin \Gamma^{*}(\Gamma^{-1}(J))$. Since $\Gamma^{*}(\Gamma^{-1}(J)) = J$, we conclude that $e \notin J$, i.e., $e \in \mathbb{K}^2 \setminus J$. 

\medskip

\textbf{1.2:} Suppose $c, d \in J$. Therefore, $\Gamma^{-1}(d), \Gamma^{-1}(c) \in \Gamma^{-1}(J)$. Denote by $p$ the point $\Gamma^{-1}(d)$. We use the labels of the points in the 8-neighborhood of $p$ as shown in Figure \ref{8-vencidadcasoc,d}. Since $\Gamma^{-1}(c)$ is 8-adjacent to $p$ but not 4-adjacent to $p$, we have $\Gamma^{-1}(c) = p_i$ for some $i \in \{1, 3, 5, 7\}$. Without loss of generality, assume $p_3 = \Gamma^{-1}(c)$. Consequently, $p_2 = \Gamma^{-1}(a)$ and $p_4 = \Gamma^{-1}(b)$ (see Figure \ref{8-vencidadcasoc,d}). By Lemma \ref{8-curva}, $\Gamma^{-1}(J)$ is an closed 8-curve, and by Lemma \ref{diferentescomponentes} (where $p$ is $\Gamma^{-1}(d)$ and $i = 3$), $p_2$ and $p_4$ belong to different 4-components of $\mathbb{Z}^2 \setminus \Gamma^{-1}(J)$. This contradicts the assumption that $a, b \in \Gamma(A)$. Therefore, this subcase is not possible.

\medskip

\textbf{1.3:} Suppose $c$ or $d$ is in $\Gamma(B)$. We will show that this subcase is not possible. Without loss of generality, assume $d \in \Gamma(B)$. Since $d \in A(a) \cap A(b)$ and $\Gamma(A) \cup \Gamma(B) = \Gamma(A \cup B)$, it follows that $\Gamma(A \cup B)$ is connected. Consequently, by Proposition \ref{4-caminoplano}, $A \cup B$ would be 4-connected, which is a contradiction. Thus, this subcase is not possible.

\medskip

\noindent\textbf{Case 2:} Suppose $e\in A_J\cup B_J$ and $e\notin \Gamma^{*}(A)\cup \Gamma^{*}(B)$. We will assume $e\in A_J$.
Note that $|A(e)\cap J|\leq 2$, because otherwise, it is easy to verify that $e\in \Gamma^*(\Gamma^{-1}(J))$, and consequently $e\in J$, which would imply that $J$ is not a Jordan curve. Let us consider the following subcases:

\medskip

\textbf{2.1:} Suppose $|A(e)\cap J|=1$. Without loss of generality, assume $A(e)\cap J=\{d\}$. It follows that $A(e)\setminus \{d\}\subseteq \Gamma(A)$, and then $N(e)\nsubseteq \Gamma(\Gamma^{-1}(J))\cup \{e\}$ and $cl(e)\nsubseteq \Gamma(\Gamma^{-1}(J))\cup \{e\}$, which implies that $e\notin J$, since $J=\Gamma^{*}(\Gamma^{-1}(J))$. 

\medskip 

\textbf{2.2:} Suppose $|A(e)\cap J|=2$. We will consider the following two alternatives:
        
\begin{itemize}  
\item []\textbf{2.2.1:} Suppose $A(e)\cap J=\{c,d\}$ or $A(e)\cap J=\{a,b\}$. Using reasoning similar to Case 1.3, we conclude that this alternative cannot occur.

\medskip 

\item []\textbf{2.2.2:} Suppose $A(e)\cap J=\{a,d\}$ or $A(e)\cap J=\{c,b\}$. Without loss of generality, assume $A(e)\cap J=\{a,d\}$. It follows that $N(e)\nsubseteq \Gamma(\Gamma^{-1}(J))\cup \{e\}$ and $cl(e)\nsubseteq \Gamma(\Gamma^{-1}(J))\cup \{e\}$, which implies that $e\notin J$.
\end{itemize}

\medskip 

With this, we conclude the proof that $\mathbb{K}^2\setminus J=\Gamma^{*}(A)\cup A_J\cup \Gamma^{*}(B)\cup B_J$.

Since the sets $\Gamma^{*}(A)\cup A_J$ and $\Gamma^{*}(B)\cup B_J$ are connected and their union is disconnected (by Proposition \ref{disconexos3}), it follows that $\Gamma^{*}(A)\cup A_J$ and $\Gamma^{*}(B)\cup B_J$ are components of $\mathbb{K}^2\setminus J$. Therefore, $\mathbb{K}^2\setminus J$ has two components.

Finally, we will verify the second statement. 
Finally, we will show the second statement. Let $x \in J$. We want to see that $x$ has adjacent points in each of the two components of $\mathbb{K}^2 \setminus J$. Let us consider two cases:

\medskip

\noindent\textbf{Case 1:} Suppose $x$ is pure. By Theorem \ref{Jordank}, there exist two points $a,b$ that are 4-adjacent to $\Gamma^{-1}(x)$ and belong to different 4-components of $\mathbb{Z}^2 \setminus \Gamma^{-1}(J)$, say $a \in A$ and $b \in B$. By Proposition \ref{4adyacentes}, $\Gamma(a), \Gamma(b)$ are adjacent to $x$ and are in the components $\Gamma^{*}(A) \cup A_J$ and $\Gamma^{*}(B) \cup B_J$, respectively.

\medskip 

\noindent\textbf{Case 2:} Suppose $x$ is mixed. Let $A(x) = \{a,b,c,d\}$, and since $J$ is a Jordan curve, we have $|A(x) \cap J| = 2$, and we will assume, without loss of generality, that $A(x) \cap J = \{c,d\}$. Note that $\Gamma^{-1}(c)$ and $\Gamma^{-1}(d)$ are 8-adjacent and not 4-adjacent. Let $\Gamma^{-1}(d) = p$, and we will assume without loss of generality that $\Gamma^{-1}(c) = p_3$. Then we have $\Gamma^{-1}(a) = p_2$ and $\Gamma^{-1}(b) = p_4$. Since $\Gamma^{-1}(J)$ is an closed 8-curve, by Proposition \ref{diferentescomponentes}, we have that $\Gamma^{-1}(a) = p_2$ and $\Gamma^{-1}(b) = p_4$ belong to different 4-components of $\mathbb{Z}^2 \setminus \Gamma^{-1}(J)$. It follows that $a$ and $b$ are in different components of $\mathbb{K}^2 \setminus J$ and $a, b \in A(x)$.

\end{proof}

\bigskip

Our final results address curves for which $\Gamma^{*}(\Gamma^{-1}(J)) \setminus J$ is not empty. However, a complete solution to this problem remains elusive, and we focus solely on a particular case.

Consider the set:
\[
S_J = \{m \in \Gamma^{*}(\Gamma^{-1}(J)) \mid m \text{ is mixed and } \left| A(m) \cap J \right| = 3 \}.
\]
We conjecture that our approach is applicable to Jordan curves satisfying $\Gamma^{*}(\Gamma^{-1}(J)) = J \cup S_J$. Nonetheless, we have only established this result in the specific case where $S_J$ contains a single element. 

\begin{propo}
\label{J_1_Jordan}
Let $J$ be a Jordan curve with more than four points and let $m \in S_J$. Let $a,b,c$ be such that $A(m) \cap J = \{a, b, c\}$. Suppose that $a$ and $b$ are closed points and $c$ is an open point (analogously, if $a$ and $b$ are open and $c$ is closed). Then:

\begin{enumerate}
\item[(i)] $J_m = J \setminus \{c\} \cup \{m\}$ is a Jordan curve.

\item[(ii)] $S_{J_m} \subseteq S_J$ and $m \notin S_{J_m}$.
\end{enumerate}
\end{propo}

\begin{proof}
(i) Observe that $J_m$ is connected.  To show that $J$ is a Jordan curve we need to verify that, for all $x \in J_m$, the following  holds 
$$
\lvert A(x) \cap J_m \rvert = 2.
$$
Notice that we only need to verify this for $x=m,a,b$. 
For $x=m$, let $d$ be such that $A(m) = \{a, b, c, d\}$. Since $\{a, b, c, d\}$ is a Jordan curve and $\lvert J \rvert > 5$, it follows that $d \notin J$. Therefore, $d \notin J_m$ and thus  $\lvert A(m) \cap J_m \rvert = 2$. 

For $x=a$. Since $a \in J$, we have that $A(a) \cap J = \{y, c\}$. Note that $y \in J_m$, $c \notin J_m$, and $a \in A(m)$. It follows then that $\lvert A(a) \cap J_m \rvert = 2$. Analogously, we have that $\lvert A(b) \cap J_m \rvert = 2$. 

\medskip 

(ii)  From (i), we know that  $\lvert A(m) \cap J_m \rvert = 2$.
Thus, $m \notin S_{J_m}$.
To see that $S_{J_m} \subseteq S_J$, let $x \in S_{J_m}$ with $x \neq m$. By the definition, of $S_{J_m}$, $x$ is a mixed point and it satisfies that $A(x) \cap J_m = \{y, w, p\}$. Given that $J_m \subseteq J \cup \{m\}$, it follows that $y, w, p \in J$, since $y, w, p$ are pure points and $m$ is mixed. Therefore, $x \in S_J$, because $A(x) \cap J = \{y, w, p\}$.

\end{proof}

Our last result is an extension of Theorem \ref{mixtoconcodiccion}. An example of the type of curve considered is given in Figure \ref{forma_1}.

\begin{figure}

\caption{A curve as in  Theorem \ref{dos_compunentes_1}.}
\label{forma_1}
\end{figure}

\begin{teo}
\label{dos_compunentes_1}
Let $J$ be a Jordan curve such that $S_J=\{m\}$ and $\Gamma^{*}(\Gamma^{-1}(J)) = J\cup\{m\}$.  Then  $\mathbb{K}^{2} \setminus J$ has two components.
\end{teo}

\begin{proof} Let  $A(m)=\{a,b,c,d\}$  such that $A(m)\cap J=\{a,b,c\}$ where  $a$ and  $b$ are  both open or closed and $J_m = J \setminus \{c\} \cup \{m\}$. By Proposition \ref{J_1_Jordan}, $J_m$ is a Jordan curve. We claim that   $\Gamma^{*}(\Gamma^{-1}(J_m)) = J_m$. For the proposition \ref{1contenencia}, we have that \(J_m \subseteq \Gamma^{*}\left(\Gamma^{-1} \left(J_m\right)\right)\). Note that 
$$
\Gamma^{*}\left(\Gamma^{-1} \left(J_m\right)\right)\subseteq \Gamma^{*}(\Gamma^{-1}(J))= J\cup\{m\}=J_m\cup\{c\}.
$$ 
We only need to verify that $c \not\in \Gamma^{*}\left(\Gamma^{-1} \left(J_m\right)\right)$, and this is true because $c$ is pure. 
Thus, we conclude that $ J_m = \Gamma^{*}\left(\Gamma^{-1} \left(J_m\right)\right)$.

\medskip

We will show that $\mathbb{K}^{2} \setminus J$ has two components. From Theorem \ref{mixtoconcodiccion} (and its proof), we know that $\mathbb{K}^{2} \setminus J_m$ has two connected components, which we denote as $C$ and $D$, where
$$
C = \Gamma^{*}(A) \cup A_{J_m}, \quad D = \Gamma^{*}(B) \cup B_{J_m},
$$
and $A$ and $B$ are the two 4-components of $\mathbb{Z}^2 \setminus \Gamma^{-1}(J_m)$.

We know that $ \left| A(m) \cap J \right| = 3$, $m \notin J$, and $A(m) \setminus J_m = \{c, d\}$, where $c \in J$.
Note that
$$
\mathbb{K}^2 \setminus J = (C \cup D \cup \{m\}) \setminus \{c\}.
$$

Now we will show that either $C \cup \{m\}$ or $D \cup \{m\}$ is connected. For this, it is sufficient to show that $d \in C$ or $d \in D$.

We have that $\Gamma^{-1}(a)$ and $\Gamma^{-1}(b)$ are 8-adjacent (but not 4-adjacent), 
$\Gamma^{-1}(a), \Gamma^{-1}(b) \in \Gamma^{-1}(J_m)$ and $\Gamma^{-1}(J_m)$ is an closed 8-curve. Following the notation of  Proposition \ref{diferentescomponentes}, let $p=\Gamma^{-1}(a)$ and $p_i=\Gamma^{-1}(b)$. Then, it follows that
$$
\Gamma^{-1}(c) = p_{i+1} \quad \text{and} \quad \Gamma^{-1}(d) = p_{i-1}.
$$
By Proposition \ref{diferentescomponentes}, applied to the closed 8-curve $\Gamma^{-1}(J_m)$, we know that $\Gamma^{-1}(c)$ and $\Gamma^{-1}(d)$ belong to different 4-components  of $\mathbb{Z}^2 \setminus \Gamma^{-1}(J_m)$. Thus, we have that $d \in C$ or $d \in D$. Without loss of generality, suppose that $d \in C$, then $C \cup \{m\}$ is connected.

Now, let us verify that $D \setminus \{c\}$ is connected. Note that $D \cap A(c)$ is connected (see Figure \ref{A(c)interDcon}) and that, for $x, y \in D \setminus \{c\}$, there exists an arc $Q \subseteq D$ connecting them. Additionally, observe that if $c \in Q$, it can be replaced by one of the elements of $Q \cap A(c)$, as appropriate (see Figure \ref{A(c)interDcon}). It follows that $D \setminus \{c\}$ is connected.

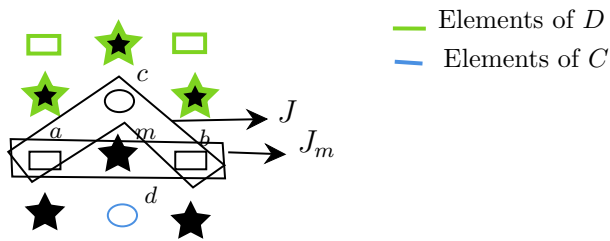
\begin{figure}
    \centering
\tikzset{every picture/.style={line width=0.75pt}} 

\begin{tikzpicture}[x=0.7pt,y=0.7pt,yscale=-1,xscale=1]

\draw   (157.81,146.55) -- (174.56,146.55) -- (174.56,156.08) -- (157.81,156.08) -- cycle ;
\draw   (236.6,146.55) -- (253.35,146.55) -- (253.35,156.08) -- (236.6,156.08) -- cycle ;
\draw  [color={rgb, 255:red, 74; green, 144; blue, 226 }  ,draw opacity=1 ][fill={rgb, 255:red, 255; green, 255; blue, 255 }  ,fill opacity=1 ][line width=0.75]  (200.16,181) .. controls (200.16,177.71) and (203.69,175.05) .. (208.04,175.05) .. controls (212.4,175.05) and (215.93,177.71) .. (215.93,181) .. controls (215.93,184.29) and (212.4,186.95) .. (208.04,186.95) .. controls (203.69,186.95) and (200.16,184.29) .. (200.16,181) -- cycle ;
\draw  [fill={rgb, 255:red, 0; green, 0; blue, 0 }  ,fill opacity=1 ] (205.58,137.88) -- (208.62,144) -- (215.42,144.98) -- (210.5,149.74) -- (211.66,156.46) -- (205.58,153.29) -- (199.5,156.46) -- (200.66,149.74) -- (195.74,144.98) -- (202.54,144) -- cycle ;
\draw  [fill={rgb, 255:red, 0; green, 0; blue, 0 }  ,fill opacity=1 ] (166.19,169.54) -- (169.23,175.66) -- (176.03,176.64) -- (171.11,181.4) -- (172.27,188.13) -- (166.19,184.95) -- (160.1,188.13) -- (161.27,181.4) -- (156.35,176.64) -- (163.14,175.66) -- cycle ;
\draw  [fill={rgb, 255:red, 0; green, 0; blue, 0 }  ,fill opacity=1 ] (244.95,174.13) -- (247.99,180.24) -- (254.79,181.23) -- (249.87,185.99) -- (251.03,192.71) -- (244.95,189.54) -- (238.86,192.71) -- (240.03,185.99) -- (235.1,181.23) -- (241.9,180.24) -- cycle ;
\draw  [color={rgb, 255:red, 126; green, 211; blue, 33 }  ,draw opacity=1 ][line width=1.5]  (157.2,83.97) -- (173.95,83.97) -- (173.95,93.49) -- (157.2,93.49) -- cycle ;
\draw  [color={rgb, 255:red, 126; green, 211; blue, 33 }  ,draw opacity=1 ][line width=1.5]  (235.99,83.17) -- (252.74,83.17) -- (252.74,92.7) -- (235.99,92.7) -- cycle ;
\draw  [color={rgb, 255:red, 0; green, 0; blue, 0 }  ,draw opacity=1 ][line width=0.75]  (198.56,119.21) .. controls (198.56,115.92) and (202.09,113.26) .. (206.45,113.26) .. controls (210.8,113.26) and (214.33,115.92) .. (214.33,119.21) .. controls (214.33,122.49) and (210.8,125.16) .. (206.45,125.16) .. controls (202.09,125.16) and (198.56,122.49) .. (198.56,119.21) -- cycle ;
\draw  [color={rgb, 255:red, 126; green, 211; blue, 33 }  ,draw opacity=1 ][fill={rgb, 255:red, 0; green, 0; blue, 0 }  ,fill opacity=1 ][line width=2.25]  (205.95,78.5) -- (208.99,84.62) -- (215.79,85.6) -- (210.87,90.36) -- (212.03,97.09) -- (205.95,93.91) -- (199.87,97.09) -- (201.03,90.36) -- (196.11,85.6) -- (202.91,84.62) -- cycle ;
\draw  [color={rgb, 255:red, 126; green, 211; blue, 33 }  ,draw opacity=1 ][fill={rgb, 255:red, 0; green, 0; blue, 0 }  ,fill opacity=1 ][line width=2.25]  (166.32,106.27) -- (169.36,112.39) -- (176.16,113.37) -- (171.24,118.13) -- (172.4,124.85) -- (166.32,121.68) -- (160.24,124.85) -- (161.4,118.13) -- (156.48,113.37) -- (163.28,112.39) -- cycle ;
\draw  [color={rgb, 255:red, 126; green, 211; blue, 33 }  ,draw opacity=1 ][fill={rgb, 255:red, 0; green, 0; blue, 0 }  ,fill opacity=1 ][line width=2.25]  (247.32,106.96) -- (250.36,113.07) -- (257.16,114.06) -- (252.24,118.82) -- (253.4,125.54) -- (247.32,122.37) -- (241.24,125.54) -- (242.4,118.82) -- (237.48,114.06) -- (244.28,113.07) -- cycle ;
\draw   (206.2,105.9) -- (263.2,153.9) -- (246.2,165.9) -- (209.2,130.9) -- (159.2,162.9) -- (146.35,146.64) -- cycle ;
\draw   (262.2,141.9) -- (263.51,161.01) -- (148.51,160.01) -- (147.65,139.75) -- cycle ;
\draw    (265.2,146.9) -- (293,147.89) ;
\draw [shift={(296,148)}, rotate = 182.05] [fill={rgb, 255:red, 0; green, 0; blue, 0 }  ][line width=0.08]  [draw opacity=0] (10.72,-5.15) -- (0,0) -- (10.72,5.15) -- (7.12,0) -- cycle    ;
\draw    (234.7,129.9) -- (283.2,128.96) ;
\draw [shift={(286.2,128.9)}, rotate = 178.89] [fill={rgb, 255:red, 0; green, 0; blue, 0 }  ][line width=0.08]  [draw opacity=0] (10.72,-5.15) -- (0,0) -- (10.72,5.15) -- (7.12,0) -- cycle    ;
\draw [color={rgb, 255:red, 126; green, 211; blue, 33 }  ,draw opacity=1 ][line width=1.5]    (354.2,79.9) -- (372,80) ;
\draw [color={rgb, 255:red, 74; green, 144; blue, 226 }  ,draw opacity=1 ][line width=1.5]    (355.2,97.9) -- (370.2,98.9) ;

\draw (213.23,132.41) node [anchor=north west][inner sep=0.75pt]  [font=\scriptsize]  {$m$};
\draw (214.03,100.49) node [anchor=north west][inner sep=0.75pt]  [font=\footnotesize]  {$c$};
\draw (218.44,163.71) node [anchor=north west][inner sep=0.75pt]  [font=\scriptsize]  {$d$};
\draw (167.02,132.1) node [anchor=north west][inner sep=0.75pt]  [font=\scriptsize]  {$a$};
\draw (248.02,132.79) node [anchor=north west][inner sep=0.75pt]  [font=\scriptsize]  {$b$};
\draw (301,133.4) node [anchor=north west][inner sep=0.75pt]    {$J_{m}$};
\draw (290,117.4) node [anchor=north west][inner sep=0.75pt]    {$J$};
\draw (377,68.4) node [anchor=north west][inner sep=0.75pt]  [font=\footnotesize]  {\text{Elements of $D$}};
\draw (380,89.4) node [anchor=north west][inner sep=0.75pt]  [font=\footnotesize]  {\text{Elements of $C$}};

\end{tikzpicture}
    \caption{$A(c)\cap D$ is connected}
    \label{A(c)interDcon}
\end{figure}

Since $D$ and $C$ are not adjacent and $m$ is not adjacent to $D \setminus \{c\}$, we conclude that $\mathbb{K}^{2} \setminus J$ has exactly two connected components.

\end{proof}


\begin{thebibliography}{100}
\bibliographystyle{plain}

\bibitem{fajardoJonard2021}
Fajardo-Rojas, D.  and Jonard-P\'erez, N.
\newblock A {J}ordan curve theorem for 2-dimensional tilings.
\newblock {\em Topology Appl.}, 300: Paper No. 107773, 18, 2021.

\bibitem{Khalimsky}
Khalimsky, E.,   Kopperman, R. y  Meyer, P. R. \textit{Computer graphics and connected topologies on finite ordered sets.} Topology and its Applications, 36(1990), no 1, 1-17.

\bibitem{Khalimsky2} 
Khalimsky, E., Kopperman, R. y Meyer, P.  \textit{Boundaries in digital planes.} Journal of Applied Mathematics and Stochastic Analysis, 3(1990), no 1, 27-55.

\bibitem{davis}
Kong,  Y.\textit{ Digital Topology. En Foundations of image understanding.} Springer Science  Business, LLC Media (2001), 73-95.


\bibitem{KongRosenfeld1989}
T.~Y. Kong and A.~Rosenfeld.
\newblock Digital topology: a comparison of the graph-based and topological
  approaches.
\newblock In {\em Topology and category theory in computer science ({O}xford,
  1989)}, Oxford Sci. Publ., pages 273--289. Oxford Univ. Press, New York,
  1991.

\bibitem{kong}
Kong, Y., y Rosenfeld, A \textit{Digital topology: Introduction and survey.} Computer Vision, Graphics, and Image Processing, 48 (1989), no 3, 357-393.

\bibitem{Richmond}
Richmond, T.
\newblock {\em General Topology: An Introduction}.
\newblock Walter de Gruyter GmbH \& Co KG, 2020.


\bibitem{R1}
Rosenfeld, A. \textit{Arcs and curves in digital pictures.} Journal of the ACM (JACM), 20 (1973), no 1, 81-87.
 
\bibitem{R2} 
Rosenfeld, A. \textit{Connectivity in digital pictures}. Journal of the ACM (JACM), 17(1970), no 1, 146-160.

\bibitem{slapal2004}
\v{S}lapal, J.
\newblock A digital analogue of the {J}ordan curve theorem.
\newblock {\em Discrete Appl. Math.}, 139(1-3): 231--251, 2004.

\bibitem{slapal2006}
\v{S}lapal, J.
\newblock Digital {J}ordan curves.
\newblock {\em Topology Appl.}, 153(17): 3255--3264, 2006.

\end{thebibliography}
\end{document}